\pdfoutput=1
\documentclass{louloupart1}
\usepackage{mathtools}
\usepackage{graphicx}
\usepackage{hyperref}
\usepackage{bm}
\usepackage{tikz}
\usetikzlibrary{shapes.geometric,calc}
\usetikzlibrary{decorations.markings,arrows,arrows.meta,positioning,cd}
\tikzset{->-/.style={decoration={
  markings,
  mark=at position #1 with {\arrow{Computer Modern Rightarrow[length=5pt,width=5pt]}}},postaction={decorate}}}
\tikzset{->-rev/.style={decoration={
  markings,
  mark=at position #1 with {\arrow{Computer Modern Rightarrow[length=5pt,width=5pt,reversed]}}},postaction={decorate}}}

\title[Property $R_\infty$ for new classes of Artin groups]{Property $R_\infty$ for new classes of Artin groups}
\author[I Soroko]{Ignat Soroko}
\givenname{Ignat}
\surname{Soroko}
\address{Ignat Soroko, Department of Mathematics, Southern Methodist University, Dallas, TX 75205, USA}
\email{isoroko@smu.edu}
\author[N Vaskou]{Nicolas Vaskou}
\givenname{Nicolas}
\surname{Vaskou}
\address{Nicolas Vaskou, Section de math\'ematiques, Universit\'e de Gen\`eve, Rue du Conseil-G\'en\'eral 7--9 1205, Gen\`eve, Switzerland}
\email{nicolas.vaskou@unige.ch}
\date{\today}

\newtheorem{thm}{Theorem} 
\newtheorem{lem}[thm]{Lemma}

\newtheorem*{delzantlemB}{Delzant's Lemma}
\newtheorem*{penners}{Penner's Construction}
\newtheorem{prop}[thm]{Proposition}
\newtheorem{corl}[thm]{Corollary}

\newtheorem*{conjR}{Conjecture (R)}
\newtheorem*{conjNA}{Conjecture (NA)}
		
\theoremstyle{definition}

\newtheorem{defin}[thm]{Definition}

\newtheorem{rem}[thm]{Remark}

\newtheorem*{acknow}{Acknowledgments}

\numberwithin{equation}{section}

\begin{document}

\def\N{\mathbb N} \def\Inn{{\rm Inn}} \def\Out{{\rm Out}} \def\Z{\mathbb Z}
\def\id{{\rm id}} \def\supp{{\rm supp}} 
\renewcommand{\Im}{\operatorname{Im}} 
\def\Ker{{\rm Ker}} \def\PP{\mathcal P} \def\Homeo{{\rm Homeo}}
\def\SHomeo{{\rm SHomeo}} \def\LHomeo{{\rm LHomeo}}
\def\MM{\mathcal M} \def\CC{\mathcal C} \def\AA{\mathcal A}
\def\S{\mathbb S} \def\FF{\mathcal F} \def\SS{\mathcal S}
\def\LL{\mathcal L} \def\D{\mathbb D} 

\newcommand{\Cyl}{\operatorname{Cyl}}
\newcommand{\Aut}{\operatorname{Aut}}
\newcommand{\conj}{\operatorname{conj}}
\newcommand{\card}{\operatorname{Card}}
\newcommand{\Sym}{\operatorname{Sym}}
\newcommand{\Cent}{\operatorname{Cent}}
\newcommand{\Fix}{\operatorname{Fix}}
\newcommand{\Isom}{\operatorname{Isom}}
\newcommand{\ov}{\overline}
\mathchardef\mhyphen="2D
\renewcommand{\le}{\leqslant}
\renewcommand{\ge}{\geqslant}
\newcommand{\lk}{\operatorname{lk}}

\begin{abstract}
We establish property $R_\infty$ for Artin groups of spherical type $D_n$, $n\ge6$, their central quotients, and also for large hyperbolic-type free-of-infinity Artin groups and some other classes of large-type Artin groups. The key ingredients are recent descriptions of the automorphism groups for these Artin groups and their action on suitable Gromov-hyperbolic spaces. We also provide a detailed proof of Delzant's Lemma, an important technical tool used in our work and in several other papers on the $R_\infty$ property.

\smallskip\noindent
{\bf 2020 Mathematics Subject Classification~} 
Primary 20F36; Secondary 20F65, 57K20, 20E36, 20E45, 57M07.

\smallskip\noindent
{\bf Keywords\ \ } 
Artin groups, large-type, free-of-infinity, property $R_\infty$, twisted conjugacy.

\end{abstract}

\maketitle


\section{Introduction}\label{sec1}

Let $G$ be a group and let $\varphi$ be an automorphism of $G$. We say that elements $g,h\in G$ are \emph{$\varphi$-twisted conjugate} if there exists $x\in G$ such that $h=x\,g\, \varphi(x)^{-1}$. This defines an equivalence relation on $G$, and the (possibly infinite) number of its equivalence classes is called the \emph{Reidemeister number} of $\varphi$, denoted by $R(\varphi)$. We say that~$G$ \emph{has property $R_{\infty}$} if $R(\varphi)=\infty$ for all $\varphi\in \Aut(G)$. 

The notion of twisted conjugacy arises in Nielsen--Reidemeister fixed point theory, where under certain natural conditions the Reidemeister number serves as an upper bound for a homotopy invariant called the Nielsen number. Also, twisted conjugacy classes appear naturally in Arthur--Selberg theory, Galois cohomology, the twisted Burnside--Frobenius theory, nonabelian cohomology, and in some topics of algebraic geometry. See~\cite{TabWon1,FelTro2} and references therein.

The problem of determining which groups have property $R_\infty$ was started in~\cite{FelHil1}, and has since been an area of active research. The list of groups known to have property $R_\infty$ is quite large and contains non-elementary Gromov-hyperbolic and relatively hyperbolic groups, non-abelian generalized Baumslag--Solitar groups, many weakly branch groups, many arithmetic linear groups, mapping class groups of surfaces with large enough complexity, and some other non-amenable groups, see~\cite{FelTro2,FelNas1} for references. On the other hand, the free nilpotent group $N_{r,c}$ of rank $r$ and the nilpotency class $c$ has property $R_\infty$ if and only if $c\ge2r$, see~\cite{DekGon1}. For some other recent developments, see~\cite{DekLat1,SgSiVe1,Troit0,Troit1,IMSWFF1}.

Let $S$ be a finite set. A \emph{Coxeter matrix} over $S$ is a symmetric matrix $(m_{st})_{s,t\in S}$ with entries in $\{1,2,\dots\}\cup\{\infty\}$, such that $m_{ss}=1$ for all $s\in S$ and $m_{st}=m_{ts}\ge2$ if $s\ne t$. 
A Coxeter matrix can be encoded by the corresponding \emph{Coxeter graph} $\Gamma$ having $S$ as the set of vertices. Two distinct vertices $s,t\in S$ are connected with an edge in $\Gamma$ if $m_{st}\ge3$, and this edge is labeled with $m_{st}$ if $m_{st}\ge4$.
The \emph{Artin group of type $\Gamma$} is the group $A[\Gamma]$ given by the presentation: 
\[
A[\Gamma]=\langle S\mid \Pi(s,t,m_{st})=\Pi(t,s,m_{ts}), \text{ for all } s\ne t,\,m_{st}\ne\infty\rangle,
\]
where $\Pi(s,t,m_{st})$ is the word $stst\dots$ of length $m_{st}\ge2$. The \emph{Coxeter group $W[\Gamma]$ of type $\Gamma$} is the quotient of $A[\Gamma]$ by all relations of the form $s^2=1$, $s\in S$. We denote by $P[\Gamma]$ the kernel of the natural epimorphism $A[\Gamma]\to W[\Gamma]$. It is called the \emph{pure Artin group of type $\Gamma$}. An Artin group $A[\Gamma]$ is called \emph{spherical}, if $W[\Gamma]$ is finite.

Concerning Artin groups, property $R_\infty$ was established for braid groups (i.e.\ Artin groups of type~$A_n$) in \cite{FeGoDa1}, for pure braid groups $P[A_n]$ in~\cite{DeGoOc1}, for some classes of large-type Artin groups in~\cite{Juhas1} (see the discussion of these results in Remark~\ref{rem:juh} of Section~\ref{sec4}) and in~\cite{JoMaSa1}, for (non-abelian) right-angled Artin groups in~\cite{Witdo1} (see also \cite{DekSen1}), and for the spherical Artin groups of types $B_n$, $D_4$, $I_2(m)$ for $m\ne2$, their pure subgroups, and for the affine Artin groups of types $\tilde A_n$, $\tilde C_n$, in~\cite{CalSor1}.

In this article we extend the list of groups having property $R_\infty$ to include Artin groups of a few new classes. The first one is the class of Artin groups of spherical type $D_n$, for $n\ge6$, for which the Coxeter graph of type $D_n$ is depicted in Figure~\ref{fig:dn}. 
Namely, we prove the following theorem.

\begin{thm}\label{thm:1}
Let $n\ge6$. Then the Artin group $A[D_n]$ and its central quotient $A[D_n]/Z(A[D_n])$ have property $R_\infty$.
\end{thm}

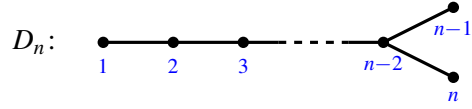
\begin{figure}
\begin{center}
\begin{tikzpicture}[scale=0.7, very thick]
\begin{scope}[scale=1.33] 
\fill (0,0) circle (2.3pt) node [below=2pt,blue] {\scriptsize$1$}; 
\fill (1,0) circle (2.3pt) node [below=2pt,blue] {\scriptsize$2$}; 
\fill (2,0) circle (2.3pt) node [below=2pt,blue] {\scriptsize$3$}; 
\fill (4,0) circle (2.3pt) node [below=1pt,blue] {\scriptsize$n{-}2$}; 
\fill (5,0.5) circle (2.3pt) node [below=1pt,blue] {\scriptsize$n{-}1$}; 
\fill (5,-0.5) circle (2.3pt) node [below=1pt,blue] {\scriptsize$n$}; 
\draw (5,0.5)--(4,0)--(5,-0.5);
\draw (0,0)--(1,0)--(2,0)--(2.5,0);
\draw [dashed] (2.5,0)--(3.5,0);
\draw (3.5,0)--(4,0);
\draw (-1,0) node {$D_n$:};	
\end{scope}
\end{tikzpicture}
\caption{The Coxeter graph of type $D_n$, $n\ge4$.\label{fig:dn}}
\end{center}
\end{figure}

The proof is based on the approach used in~\cite{FeGoDa1} and \cite{CalSor1} to establish property $R_\infty$ for some Artin groups whose automorphism groups can be embedded into mapping class groups of certain surfaces with punctures. 
The new ingredient that made this approach possible for Artin groups of type $D_n$ is the recent description of automorphisms of Artin groups of type $D_n$ for $n\ge6$, obtained in~\cite{CasPar1}. Using this result, we establish an embedding of $\Aut(A[D_n])$ into the extended mapping class group of a suitable punctured surface, in Proposition~\ref{prop:mcg}, which may be of independent interest. We use this embedding to show that $\Aut(A[D_n])$ acts in a non-elementary way on the Gromov-hyperbolic complex of curves for the surface, and then applying Delzant's Lemma (see Section~\ref{sec2}) allows us to detect infinitely many twisted conjugacy classes.

We note that property $R_\infty$ for the Artin group of type $D_4$ was established in~\cite{CalSor1}. 
The only remaining case in the family $A[D_n]$ is $A[D_5]$. At present we do not have a description of $\Aut(A[D_5])$, which is the key input in our approach; we conjecture though that both $A[D_5]$ and $A[D_5]/Z(A[D_5])$ have property $R_\infty$ as well.

The second class of Artin groups for which we establish property $R_\infty$ is the class of \emph{large-type hyperbolic-type free-of-infinity} Artin groups. An Artin group $A[\Gamma]$ is called \emph{large} (or \emph{large-type}) if all $m_{st}\ge3$ (with  $m_{st}=\infty$ possible) for all $s,t\in S$; $A[\Gamma]$ is called \emph{hyperbolic-type} if its Coxeter group $W[\Gamma]$ is word-hyperbolic, and $A[\Gamma]$ is called \emph{free-of-infinity}, if $m_{st}\ne\infty$ for all $s,t\in S$. We note that for the class of large-type Artin groups, being hyperbolic-type is equivalent to the absence of triangles with edge labels $(3,3,3)$ in its Coxeter graph $\Gamma$ (see~\cite[Corollary~12.6.3]{Davis1}). We also note that Artin groups of hyperbolic-type are not word-hyperbolic themselves, unless they are free. We prove the following theorem (see Corollary~\ref{cor:ltfoiht}).
 
\begin{thm}\label{thm:2}
Let $A[\Gamma]$ be a large hyperbolic-type free-of-infinity Artin group. Then it has property~$R_\infty$.
\end{thm}

We also establish property $R_\infty$ for some other subclasses of large hyperbolic-type Artin groups, whose definitions are deferred to Section~\ref{sec4}, in the following theorem (whose proof is given in Corollaries~\ref{cor:xxxl} and \ref{cor:hierarchy}).

\begin{thm}\label{thm:3}
Let $A[\Gamma]$ be a large hyperbolic-type Artin group. If either
\begin{itemize}
\item $A[\Gamma]$ has XXXL type and $\Gamma$ is twistless, or
\item $\Gamma$ admits a
twistless hierarchy terminating in twistless stars,
\end{itemize}
then $A[\Gamma]$ has property~$R_\infty$.
\end{thm}

The proofs of Theorems~\ref{thm:2} and~\ref{thm:3} are based on the recent description of the automorphism groups for these classes of Artin groups~\cite{Vasko2,BlMaVa1,HuOsVa1}, and on their action 
on the Deligne complex, which is known to be Gromov-hyperbolic.

This paper is organized as follows. In Section~\ref{sec2} we give preliminary results necessary for proving Theorem~\ref{thm:1} in Section~\ref{sec3} and Theorems~\ref{thm:2} and \ref{thm:3} in Section~\ref{sec4}. 
In Section~\ref{app} we provide, for the reader's convenience, a detailed proof of Delzant’s Lemma, which plays a key role in establishing the 
$R_\infty$ property in a number of papers and is stated in~\cite[Lemma~3.4]{LevLus1} with a brief proof sketch.
Finally, in Section~\ref{sec5} we recall some conjectures concerning property $R_\infty$, which we believe should be more broadly known. 

\begin{acknow}
The authors thank Evgenij Troitsky for his comments on the status of the conjectures discussed in Section~\ref{sec5}, and Matthieu Calvez, R\'emi Coulon, Mark Hagen, Gilbert Levitt, Alexandre Martin, and Peter Wong for useful discussions. The authors also thank Jennifer Taback for sharing the preprint~\cite{TabWhy1} with them.
This work was started at the AIM workshop ``Geometry and topology of Artin groups'' organized by Ruth Charney, Kasia Jankiewicz and Kevin Schreve in September 2023, and continued at
the BIRS workshop ``Combinatorial Nonpositive Curvature'' organized by Kasia Jankiewicz and Piotr Przytycki in September 2024.
The authors thank the organizers, the AIM and BIRS.
The first author acknowledges support from the AMS--Simons travel grant. The second author is supported by the Postdoc.Mobility fellowship \#P500PT\textunderscore 210985 of the Swiss National Science Foundation.
\end{acknow}

\section{Preliminaries on twisted conjugacy}\label{sec2}
Recall that for a group $G$ with center $Z(G)$, $\Inn(G)$ denotes the subgroup of $\Aut(G)$ consisting of all inner automorphisms. For $g\in G$, we denote by $\conj_g$ the inner automorphism of conjugation by $g$:  $x\mapsto gxg^{-1}$. The map $\conj\colon g\mapsto \conj_g$ identifies $\Inn(G)$ with $G/Z(G)$ and $\Inn(G)$ is a normal subgroup of $\Aut(G)$. The \emph{outer automorphism group} of $G$ is the quotient $\Out(G)=\Aut(G)/\Inn(G)$.

To determine property $R_\infty$ for a group $G$ one can proceed as follows.

\emph{Step 1: Reduce the problem to $G/Z(G)$:}

\begin{lem}[See e.g.~\protect{\cite[Section~2]{FeGoDa1}}]\label{lem:epi}
Let $G$ be a group, $\varphi$ be an automorphism of $G$, and $H$ be a normal $\varphi$-invariant subgroup of $G$. Denote by $\bar\varphi$ the automorphism induced by $\varphi$ on $G/H$. Then the condition $R(\bar\varphi)=\infty$ implies $R(\varphi)=\infty$. In particular, if $G/Z(G)$ has property $R_{\infty}$, then so does~$G$. \qed 
\end{lem}

\emph{Step 2: Reduce the problem to detecting the usual (non-twisted) conjugacy classes:}
 
Let $G$ be a group with $Z(G)=1$ and let $\varphi$ be an automorphism of $G$. Let $m\in\{1,2,3,\dots\}\cup\{\infty\}$ 
be the order of $\varphi$ in $\Out(G)$; if $m<\infty$, let $p\in G$ be the unique element such that $\varphi^m=\conj_p$. 
Consider the group
\[
G_\varphi=
\begin{cases} 
G*\langle t\rangle/\langle tgt^{-1}=\varphi(g) \text{ for all }g\in G\,\rangle,  & \text{if $m=\infty$,}\\
G*\langle t\rangle/\langle tgt^{-1}=\varphi(g) \text{ for all }g\in G,\,t^m=p\,\rangle,  & \text{if $m<\infty$.} 
\end{cases}
\]
We have the following useful properties of $G_\varphi$:
\begin{lem}[\protect{Calvez--Soroko~\cite[Lemma~3]{CalSor1}}]\label{lem:embeds}
Let $G$ be a group with $Z(G)=1$, $\varphi\in\Aut(G)$, and $G_\varphi$ defined as above. Then $G$ is a normal subgroup of $G_\varphi$, the quotient $G_{\varphi}/G$ is cyclic, and the assignment $g\mapsto \conj_g$ and $t\mapsto \varphi$ defines an injective homomorphism from $G_{\varphi}$ to $\Aut(G)$. \qed
\end{lem}
We observe that if $m=\infty$, $G_\varphi$ is the semidirect product $G\rtimes_\varphi\Z$, and if $m<\infty$, we have a 
commutative diagram:
\[
\setlength\mathsurround{0pt}
	\begin{tikzcd}
		1\arrow[r] & G\arrow[r]\arrow[d,"\simeq"] & G_\varphi\arrow[r]\arrow[d,hook] & \Z/m\Z \arrow[r]\arrow[d,hook] & 1\\
		1\arrow[r] & \Inn(G)\arrow[r] & \Aut(G)\arrow[r] & \Out(G)\arrow[r] & 1 
	\end{tikzcd}
\]
in which the top short exact sequence does not split in general (for example, if $G$ is the free group of rank $2$, there are elements of order~$6$ in $\Out(G)$, but not in $\Aut(G)$, see~\cite[Proposition~4.6]{LynSch1}). 
We also note that if $m=1$ then $G_\varphi$ is isomorphic to $G$. 

It turns out that $\varphi$-twisted conjugacy classes in $G$ bijectively correspond to the usual conjugacy classes in the coset $Gt$ of $G_\varphi$:

\begin{lem}[\protect{Calvez--Soroko~\cite[Proposition~4]{CalSor1}}]\label{lem:cosets}
Let $G$ be a group with $Z(G)=1$, $\varphi\in\Aut(G)$, and $G_\varphi$ defined as above. Two elements $g,h\in G$ are $\varphi$-twisted conjugate in $G$ if and only if the elements $gt$ and~$ht$ of $G_{\varphi}$ are conjugate in $G_{\varphi}$. In particular, 
$R(\varphi)=\infty$ if and only if the coset $Gt$ in $G_{\varphi}$ contains infinitely many conjugacy classes.\qed
\end{lem}

\emph{Step 3: Detect infinitely many conjugacy classes using Delzant's Lemma:}

In case when a group $\Gamma$ acts by isometries in a non-elementary way on a Gromov-hyperbolic space, the very useful lemma of Delzant allows us to detect infinitely many conjugacy classes under certain conditions, which will be satisfied in our case. We recall the relevant definitions.

Let $\Gamma$ be a group acting by isometries on a geodesic Gromov-hyperbolic space $X$ (not necessarily proper). For any $x\in X$ consider the set~$\Lambda(\Gamma)$ of accumulation points of the orbit $\Gamma x$ on the boundary $\partial X$. The action of $\Gamma$ on $X$ is called \emph{non-elementary}, if $\Lambda(\Gamma)$ contains at least three points; see~\cite[Section~3]{Osin1}.

The following result is our key tool for detecting infinitely many twisted conjugacy classes. We provide an detailed proof of it in Section~\ref{app}.

\begin{delzantlemB}[\protect{\cite[Lemma 6.3]{FeGoDa1}}, \protect{\cite[Lemma 3.4]{LevLus1}}]\label{lem:delzant}
Let $\Gamma$ be a group acting non-elementarily by isometries on a geodesic Gromov-hyperbolic space (which is not assumed to be proper), and let $K$ be a normal subgroup of $\Gamma$ such that the quotient $\Gamma/K$ is abelian. Then every coset of $K$ contains infinitely many conjugacy classes. 
\end{delzantlemB}

\begin{corl}\label{cor:7}
Let $G$ be a group with $Z(G)=1$. If $\Aut(G)$ acts by isometries on a Gromov-hyperbolic space $X$ in such a way that the action of the subgroup $G\simeq\Inn(G)\le\Aut(G)$ is non-elementary, then $G$ has property $R_\infty$.
\end{corl}
\begin{proof}
By Lemma~\ref{lem:embeds}, for any $\varphi\in\Aut(G)$ we have embeddings $G\le G_\varphi\le \Aut(G)$, which imply that $G_\varphi$ also acts on $X$ non-elementarily. We apply Delzant's Lemma to $\Gamma=G_\varphi$, $K=G$, and conclude that the coset $Gt$ of $G_\varphi$ contains infinitely many conjugacy classes. Lemma~\ref{lem:cosets} now implies that $G$ has property $R_\infty$.
\end{proof}

\begin{rem}
Instead of $G_\varphi$, some authors work with the (unrestricted) mapping torus $G\rtimes_\varphi\Z$ in order to translate
$\varphi$-twisted conjugacy in $G$ into ordinary conjugacy in $G\rtimes_\varphi\Z$
(see e.g.~\cite{IMSWFF1,Juhas1}).
Indeed, for $G\rtimes_\varphi\Z$ the same criterion holds as for $G_\varphi$: arbitrary elements $g,h\in G$ are $\varphi$-twisted conjugate in $G$ if and only if $gt$ and $ht$
are conjugate in $G\rtimes_\varphi\Z$. Following~\cite{FeGoDa1} and ~\cite{CalSor1}, we are working with $G_\varphi$, since it comes with a canonical embedding into $\Aut(G)$, which allows us to make $G_\varphi$ act on the same Gromov-hyperbolic space on which $\Aut(G)$ acts. By contrast, when $\varphi$ has finite order $m$ in $\Out(G)$, so that $\varphi^m=\conj_p$ for some $p\in G$,
the natural homomorphism $G\rtimes_\varphi\Z\to\Aut(G)$ given by $g\mapsto \conj_g$ and $t\mapsto\varphi$
is not injective (its kernel contains $t^m p^{-1}$). 
\end{rem}

\section{A geometric representation of \texorpdfstring{$\Aut(A[D_n])$}{Aut(A[Dn])}}\label{sec3}

In this section, we describe an embedding of $\Aut(A[D_n])$ into the (extended) mapping class group of a suitable surface with marked points, and use it to prove Theorem~\ref{thm:1}. 

First, we introduce some notation. We denote by $t_1,\dots,t_n$ the standard generators of $A=A[D_n]$ corresponding to the numbering of vertices of the Coxeter graph in Figure~\ref{fig:dn}. The element
\[
\Delta = (t_1\dots t_{n-2}t_{n-1}t_nt_{n-2}\dots t_1)(t_2\dots t_{n-2}t_{n-1}t_nt_{n-2}\dots t_2)\dots (t_{n-2}t_{n-1}t_nt_{n-2})(t_{n-1}t_n)
\]
is the so-called Garside element of $A$ (see~\cite[Lemma~5.1]{Paris1}). If $n$ is even, then $\Delta t_i\Delta^{-1}=t_i$ for all $1\le i\le n$, and the center $Z(A)$ is generated by $\Delta$. If $n$ is odd, then $\Delta t_i\Delta^{-1}=t_i$ for all $1\le i\le n-2$, $\Delta t_{n-1}\Delta^{-1}=t_n$, $\Delta t_n\Delta^{-1}=t_{n-1}$, and $Z(A)$ is generated by $\Delta^2$, see~\cite[Satz~7.2]{BriSai1}.

We define automorphisms $\zeta,\chi\in\Aut(A[D_n])$ as:
\begin{gather*}
\zeta(t_i)=t_i\text{\quad for\,\,\,}1\le i\le n-2,\quad \zeta(t_{n-1})=t_n,\quad\zeta(t_n)=t_{n-1},\\
\chi(t_i)=t_i^{-1}\text{\quad for\,\,\,}1\le i\le n.
\end{gather*}
Clearly, $\zeta$ and $\chi$ commute and both have order $2$, and hence they generate a subgroup of $\Aut(A[D_n])$ isomorphic to $\Z/2\Z\times\Z/2\Z$. If $n$ is odd, then $\zeta$ is the conjugation automorphism by $\Delta$. If $n$ is even, then $\zeta$ is not inner, as can be seen from the proof of Satz~7.1 of~\cite{BriSai1}. On the other hand, the automorphism $\chi$ is never inner. (Indeed, consider the homomorphism $\xi\colon A[D_n]\to\Z$, $t_i\mapsto1$ for all $1\le i\le n$. For every inner automorphism $\phi$, we have $\xi(\phi(t_i))=\xi(t_i)$ for all $i$. For the automorphism $\chi$ we have: $\xi(\chi(t_i))=-1=-\xi(t_i)$ for all $i$.)

Denote for brevity $\ov{A[D_n]}=A[D_n]/Z(A[D_n])$ and let $\pi\colon A[D_n]\to\ov{A[D_n]}$ be the canonical projection. It is well known (and easy to prove) that $Z(\ov{A[D_n]})=1$. For each $1\le i\le n$, we set $\bar t_i=\pi(t_i)$ and denote $\bar\zeta$, $\bar\chi$ the automorphisms of $\ov{A[D_n]}$ induced by $\zeta$, $\chi$, respectively. In~\cite{CasPar1}, Castel and Paris obtained the following description of $\Aut(A[D_n])$ and $\Aut(\ov{A[D_n]})$.
\begin{thm}[\protect{Castel--Paris~\cite[Corollary~2.10]{CasPar1}}]\label{thm:autdn}
Let $n\ge6$. 
\begin{enumerate}
\item If $n$ is even, then
\[
\Aut(\ov{A[D_n]})=\Inn(\ov{A[D_n]})\rtimes\langle\bar\zeta,\bar\chi\rangle\simeq\ov{A[D_n]}\rtimes(\Z/2\Z\times\Z/2\Z),
\]
and $\Out(\ov{A[D_n]})\simeq \Z/2\Z\times\Z/2\Z$.
\item If $n$ is odd, then
\[
\Aut(\ov{A[D_n]})=\Inn(\ov{A[D_n]})\rtimes\langle\bar\chi\rangle\simeq\ov{A[D_n]}\rtimes\Z/2\Z,
\]
and $\Out(\ov{A[D_n]})\simeq \Z/2\Z$.
\end{enumerate}
Moreover, for any $n\ge6$ we have $\Aut(A[D_n])\simeq\Aut(\ov{A[D_n]})$ and $\Out(A[D_n])\simeq\Out(\ov{A[D_n]})$.\qed
\end{thm}

We now recall the relevant notions related to mapping class groups of surfaces.

Let $\Sigma$ be an orientable surface with or without boundary, and let $\PP$ be a finite collection of different points in the interior of $\Sigma$. The \emph{mapping class group} $\MM(\Sigma,\PP)$ of the pair $(\Sigma,\PP)$ is the group of orientation-preserving homeomorphisms of $\Sigma$, identical on the boundary and permuting the set $\PP$, considered up to isotopies identical on the boundary and fixing $\PP$ pointwise. If we allow orientation-reversing homeomorphisms in the above definition, we get the notion of the \emph{extended mapping class group} of the pair $(\Sigma,\PP)$, which is denoted $\MM^*(\Sigma,\PP)$. If a surface $\Sigma$ has nonempty boundary, then $\MM^*(\Sigma,\PP)=\MM(\Sigma,\PP)$, otherwise $\MM(\Sigma,\PP)$ is a subgroup of index $2$ in $\MM^*(\Sigma,\PP)$. The \emph{pure mapping class group} of the pair $(\Sigma,\PP)$ is the finite index subgroup $\PP\MM(\Sigma,\PP)$ of $\MM(\Sigma,\PP)$ which fixes the set $\PP$ pointwise. These groups fit into the short exact sequence:
\[
1\longrightarrow \PP\MM(\Sigma,\PP)\longrightarrow\MM(\Sigma,\PP)\longrightarrow\mathfrak S_N\longrightarrow 1,
\]
which does not split in general. Here $N=|\PP|$ is the cardinality of $\PP$ and $\mathfrak S_N$ is the symmetric group on $N$ letters. We refer the reader to~\cite{FarMar1} for more information on the mapping class groups.

Let $n\ge3$ and $\Sigma_n$ be the orientable surface without boundary of genus $\lfloor (n-1)/2\rfloor$, where $\lfloor x\rfloor$ denotes the largest integer which is less than or equal to $x$, and let $\PP_n$ be a finite set of punctures in $\Sigma_n$, with $|\PP_n|=2$, if $n$ is odd, and $|\PP_n|=3$, if $n$ is even, see Figure~\ref{fig:sigma}.

\begin{figure}[!hbt]
\begin{center}
\include{TikzForSigma}
\end{center}
\caption{Surface $(\Sigma_n,\PP_n)$ with $\PP_n=\{p_1,p_2\}$ for $n$ odd and $\PP_n=\{p_1,p_2,p_3\}$ for $n$ even. The image of the standard generator $\bar t_i$ of $\ov{A[D_n]}$ under the embedding of Proposition~\ref{prop:mcg} is the Dehn twist about the circle $c_i$. The hyperelliptic involution $\iota$ rotates the surface by angle $\pi$ in the ambient space.
\label{fig:sigma}}
\end{figure}

\begin{prop}\label{prop:mcg} 
Let $n\ge6$. There exists an embedding $\ov{A[D_n]}\lhook\joinrel\longrightarrow\MM(\Sigma_n,\PP_n)$ such that every automorphism of $\ov{A[D_n]}$ is induced by a conjugation with an element of $\MM^*(\Sigma_n,\PP_n)$. In particular, we have inclusions:
\[
\ov{A[D_n]}\le \Aut(\ov{A[D_n]})\le \Inn\bigl(\MM^*(\Sigma_n,\PP_n)\bigr)\simeq\MM^*(\Sigma_n,\PP_n).
\]
\end{prop}

\begin{figure}
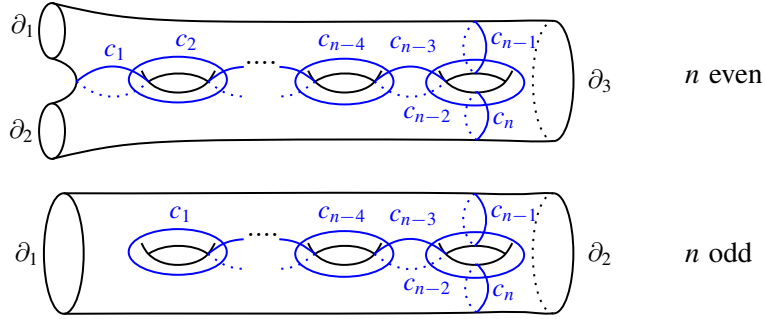

\begin{center}
\include{TikzForSigmabdry}
\end{center}
\caption{Surface $\Sigma_n^\partial$ with the images of standard generators $t_i$ of $A[D_n]$ under the Perron--Vannier embedding. (The image of $t_i$ is the Dehn twist about the circle marked $c_i$.)
\label{fig:sigmabdry}}
\end{figure}

\begin{proof}
There is a well-known embedding $\rho\colon A[D_n]\lhook\joinrel\longrightarrow\MM(\Sigma_n^\partial)$ of Perron and Vannier~\cite{PerVan1}, where $\Sigma_n^\partial$ is the surface with boundary obtained from $(\Sigma_n,\PP_n)$ by making every puncture in $\PP_n$ into a boundary component, see Figure~\ref{fig:sigmabdry}. (A more streamlined proof for a similar embedding, for the surface obtained from $(\Sigma_n,\PP_n)$ by making all but one punctures of $\PP_n$ into boundary components, was recently given in~\cite[Theorem~3.6]{CasPar1}.)

On the other hand, the inclusion of surfaces $\Sigma_n^\partial \lhook\joinrel\longrightarrow\Sigma_n$ induces the homomorphism $\eta\colon\MM(\Sigma_n^\partial)\to\MM(\Sigma_n,\PP_n)$ with kernel $\ker\eta\simeq\Z^N$ generated by boundary twists $T_{\partial_i}$, where $N=|\PP_n|\in\{2,3\}$, see~\cite[Theorem~3.18]{FarMar1}. Moreover, $\ker\eta$ lies in the center of $\MM(\Sigma_n^\partial)$, since boundary twists commute with all mapping classes in $\MM(\Sigma_n^\partial)$.

We claim that $\rho(A[D_n])\cap \ker\eta=\rho(Z(A[D_n]))$. 
Indeed, since $\ker\eta$ is central in $\MM(\Sigma_n^\partial)$, we have inclusion $\rho(A[D_n])\cap \ker\eta\subseteq\rho(Z(A[D_n]))$. On the other hand, we know the image of the generator of the center $Z(A[D_n])$ under $\rho$ by~\cite[Proposition~2.12]{LabPar1} (see also~\cite[Table~1]{Matsu1}):
\begin{gather*}
\rho(\Delta^2)=T_{\partial_1}T_{\partial_2}^{n-2},\qquad\text{if $n$ is odd;}\\
\rho(\Delta)=T_{\partial_1}T_{\partial_2}T_{\partial_3}^{n/2-1},\qquad\text{if $n$ is even}.
\end{gather*}
This shows that $\rho(Z(A[D_n]))$ lies in $\ker\eta=\langle T_{\partial_i}\rangle$, and we conclude that $\rho(A[D_n])\cap \ker\eta=\rho(Z(A[D_n]))$.

Thus we can define an embedding $\bar\rho\colon\ov{A[D_n]}\lhook\joinrel\longrightarrow\MM(\Sigma_n,\PP_n)$ induced by $\rho$:
\[
\bar\rho\colon \ov{A[D_n]}=A[D_n]/Z(A[D_n])\longrightarrow \MM(\Sigma_n^\partial)/\ker\eta=\PP\MM(\Sigma_n,\PP_n)\le\MM(\Sigma_n,\PP_n).
\]
We see that
\[
\bar\rho(\bar t_i)=T_{c_i},\quad \text{for all $1\le i\le n$},
\]
where $T_{c_i}$ are the Dehn twists about the curves $c_i$ shown in Figure~\ref{fig:sigma}. For the rest of the proof we identify $\ov{A[D_n]}$ with its image under $\bar\rho$ in $\MM(\Sigma_n,\PP_n)$.

Recall that $\Aut(\ov{A[D_n]})$ is generated by $\Inn(\ov{A[D_n]})$,  $\bar\zeta$, and $\bar\chi$, by Theorem~\ref{thm:autdn}. Naturally, any inner automorphism of $\ov{A[D_n]}\le\MM(\Sigma_n,\PP_n)$ is induced by the corresponding inner automorphism of the ambient group $\MM(\Sigma_n,\PP_n)$, hence we just need to realize $\bar \zeta$ and $\bar\chi$ as conjugations inside $\MM^*(\Sigma_n,\PP_n)$.

For automorphism $\bar\zeta$, consider the mapping class $\iota$, which (viewed on the closed surface) is called \emph{a hyperelliptic involution}, and which can be described as the rotation of the surface in the ambient space by $180$ degrees about the horizontal axis, see Figure~\ref{fig:sigma}. Clearly,
\begin{align*}
\iota \circ T_{c_i}\circ\iota &= T_{\iota(c_i)}=T_{c_i}, \text{\qquad for\,\,\, $1\le i\le n-2$,\quad and}\\
\iota \circ T_{c_{n-1}}\circ\iota &= T_{\iota(c_{n-1})}=T_{c_n},\\
\iota \circ T_{c_{n}}\circ\iota &= T_{\iota(c_{n})}=T_{c_{n-1}}.
\end{align*}
(Here we treat each $c_i$ as the isotopy class of unoriented curves containing the circle $c_i$ drawn in Figure~\ref{fig:sigma}.) We see that, indeed, the conjugation by $\iota$ induces the same automorphism on $\ov{A[D_n]}$ as $\bar\zeta$. 

If $n$ is odd, then $\iota$ is equal to $\bar\rho(\pi(\Delta))$ as mapping classes in $\MM(\Sigma_n,\PP_n)$. Indeed, from the properties of $\Delta$ mentioned in the beginning of the section, it follows that conjugating $\ov{A[D_n]}$ by $\pi(\Delta)$ induces automorphism $\bar\zeta$. Thus, $\phi=\iota\circ\bar\rho(\pi(\Delta))$ leaves all curves $c_i$ invariant (viewed as isotopy classes of unoriented curves). Now we notice that cutting the punctured surface $(\Sigma_n,\PP_n)$ along the curves $\{c_i\}_{i=1}^n$ results in two once-punctured disks such that the homeomorphism of their boundary induced by $\phi$ is identity. The curves $\{c_i\}$ fill the surface $(\Sigma_n,\PP_n)$, so we can apply the Alexander method (\cite[Proposition~2.8]{FarMar1}) and conclude that $\phi=\id$, and hence $\iota=\bar\rho(\pi(\Delta))$. This proves that, for $n$ odd, the conjugation by $\iota$ is an inner automorphism of $\bar\rho(\ov{A[D_n]})$.

If $n$ is even, then $\iota$ does not belong to $\bar\rho(\ov{A[D_n]})$, since $\iota$ interchanges punctures $p_1$ and $p_2$, whereas the image of $\bar\rho$ lies entirely in $\PP\MM(\Sigma_n,\PP_n)$, so conjugating by $\iota$ induces the outer automorphism $\bar\zeta$ of $\bar\rho(\ov{A[D_n]})$.

For automorphism $\bar\chi$, consider the orientation-reversing mapping class $X\in\MM^*(\Sigma_n,\PP_n)$ which can be described as the reflection of the surface $\Sigma_n$ in the ``plane of the picture'' in Figure~\ref{fig:sigma}. Since $X$ preserves each $c_i$ setwise but reverses orientation, we see that
\[
X\circ T_{c_i}\circ X=T^{-1}_{X(c_i)}=T^{-1}_{c_i},\qquad\text{for\ \ $1\le i\le n$}.
\]
This means that the conjugation by $X$ induces automorphism $\bar\chi$ on $\bar\rho(\ov{A[D_n]})$. Since $X$ reverses orientation, and all mapping classes from $\bar\rho(\ov{A[D_n]})$ preserve it, we see that the conjugation by $X$ is an outer automorphism of $\bar\rho(\ov{A[D_n]})$, as expected.

We check by inspection that $X\circ \iota=\iota\circ X$, and that $\iota\circ X$ does not belong to $\bar\rho(\ov{A[D_n]})$ since it reverses orientation. This means that we have the following group isomorphisms:
\begin{align*}
\Aut(\ov{A[D_n]})&\simeq
\bar\rho(\ov{A[D_n]})\rtimes\bigl(\langle \conj_\iota\rangle\times\langle \conj_X\rangle\bigr)\le\Inn\bigl(\MM^*(\Sigma_n,\PP_n)\bigr),\qquad\text{if $n$ is even,}\\
\Aut(\ov{A[D_n]})&\simeq
\bar\rho(\ov{A[D_n]})\rtimes\langle \conj_X\rangle\le\Inn\bigl(\MM^*(\Sigma_n,\PP_n)\bigr),\qquad\text{if $n$ is odd.}
\end{align*}
It remains to show the last isomorphism: $\Inn\bigl(\MM^*(\Sigma_n,\PP_n)\bigr)\simeq \MM^*(\Sigma_n,\PP_n)$. Let $z$ be a nontrivial central element in $\MM^*(\Sigma_n,\PP_n)$. By the classification of centers of $\MM(\Sigma,\PP)$ for various surfaces with punctures $(\Sigma,\PP)$ given in~\cite[p.\,77]{FarMar1}, we conclude that $z$ cannot belong to $\MM(\Sigma_n,\PP_n)$, if $n\ge4$. Hence $z$ reverses orientation. Since $z$ commutes with every Dehn twist $T_c$, we have:
\[
T_c = z\, T_c\, z^{-1} = T^{-1}_{z(c)},
\]
which contradicts the property of Dehn twists:
\[
T_a^k=T_b^\ell\quad \Longleftrightarrow\quad a=b \text{\quad and \quad} k=\ell,
\]
see~\cite[p.\,75]{FarMar1}. Hence $Z\bigl(\MM^*(\Sigma_n,\PP_n)\bigr)=1$, and the result follows.
\end{proof}

Now we are ready to prove Theorem~\ref{thm:1}.

\begin{proof}[Proof of Theorem~\ref{thm:1}]
Our goal is to apply Corollary~\ref{cor:7}. 

As was mentioned before, it is known that the center of $\ov{A[D_n]}$ is trivial. 
By Proposition~\ref{prop:mcg}, we have an embedding
\[
\ov{A[D_n]}\le\Aut(\ov{A[D_n]})\le\MM^*(\Sigma_n,\PP_n).
\]
We use the well-known fact that $\MM^*(\Sigma_n,\PP_n)$ acts by isometries on the curve complex $\CC=\CC(\Sigma_n,\PP_n)$ of the surface. This complex is Gromov-hyperbolic by the result of Masur and Minsky~\cite[Theorem~1.1]{MasMir1}, if the surface has big enough complexity, i.e.\ if $3g+p-4>0$, where $g$ denotes the genus of the surface, and $p$ the number of punctures. Recalling that $\Sigma_n$ has genus $g=\lfloor (n-1)/2\rfloor$, with $p=2$ for $n$ odd and $p=3$ for $n$ even, we see that the required inequality is satisfied already for $n\ge3$ and a fortiori for $n\ge6$.

We need to prove that the action of $\ov{A[D_n]}$ on $\CC$ is non-elementary. By~\cite[Proposition~4.6]{MasMir1}, an element $g$ of $\MM^*(\Sigma_n,\PP_n)$ acts loxodromically on $\CC$ if and only if $g$ is pseudo-Anosov. It will be sufficient to produce two pseudo-Anosov elements $g,h\in\ov{A[D_n]}$ which generate a rank $2$ free group. To produce pseudo-Anosov elements we use Penner's construction~\cite{Penne1}, \cite[Theorem~14.4]{FarMar1}:
\begin{penners}
Let $A=\{\alpha_1,\dots,\alpha_\ell\}$ and $B=\{\beta_1,\dots,\beta_m\}$ be multicurves in a surface $\Sigma$ that together fill $\Sigma$. Any product of positive powers of the $T_{\alpha_i}$ and negative powers of the $T_{\beta_i}$, where each $\alpha_i$ and each $\beta_i$ appear at least once, is pseudo-Anosov.\qed
\end{penners}
Recall that a collection of curves \emph{fills} a surface $\Sigma$, if the surface obtained from $\Sigma$ by cutting along all the curves from the collection, is a disjoint union of disks and once-punctured disks.

To choose multicurves $A$ and $B$ for application of Penner's Construction in the surface $(\Sigma_n,\PP_n)$ we notice that the family of circles $\{c_i\}_{i=1}^n$ in Figure~\ref{fig:sigma} cuts the surface $\Sigma_n$ into two or three once-punctured disks, i.e.\ it fills the surface $\Sigma_n$. This allows us to set:
\begin{align*}
A=\{c_{n},c_{n-1},c_{n-3},\dots,c_3,c_1\}, \quad
&B=\{c_{n-2},c_{n-4},\dots,c_4,c_2\},\quad\text{
if $n$ is even,\, and}\\
A=\{c_{n},c_{n-1},c_{n-3},\dots,c_4,c_2\},\quad
&B=\{c_{n-2},c_{n-4},\dots,c_3,c_1\},\quad\text{
if $n$ is odd.}
\end{align*}
Next, we introduce elements 
\[
f_A=\textstyle\prod_{c_i\in A}T_{c_i}^2,\qquad f_B=\prod_{c_j\in B}T_{c_j}^2,
\]
and set
\[
g=f_A^2\,f_B^{-1},\qquad h=f_A\,f_B^{-2}.
\]
Notice that the twists $T_{c_i}$ participating in $f_A$ commute since the curves $c_i\in A$ are disjoint, and the same observation applies for twists $T_{c_j}$ in $f_B$. Hence, by Penner's Construction, $g$ and $h$ are pseudo-Anosov. 

Now consider elements $Q_i=T_{c_i}^2$ as mapping classes on the surface $\Sigma_n^\partial$, see Figure~\ref{fig:sigmabdry}. They belong to the subgroup $A[D_n]\le\MM(\Sigma_n^\partial)$, by the Perron--Vannier embedding~\cite{PerVan1}.
By the solution of Tits' conjecture by Crisp and Paris~\cite{CriPar1}, we know that elements $Q_i$ generate a right-angled Artin subgroup $K=\langle\, Q_i\mid 1\le i\le n\,\rangle$ of $A[D_n]$, which has trivial center $Z(K)=1$, see e.g.~\cite[Section 2.3]{Charn1}. Thus, $K$ maps isomorphically into $\MM(\Sigma_n,\PP_n)$ via $\eta\colon\MM(\Sigma_n^\partial)\to\MM(\Sigma_n,\PP_n)$ (since $\ker\eta$ is central in $\MM(\Sigma_n^\partial)$, see the proof of Proposition~\ref{prop:mcg}), and $g$, $h$ belong to $\eta(K)$ by construction.

It is easy to see that $g$ and $h$ do not commute. Indeed, let $a=Q_n$, $b=Q_{n-2}$, and consider an epimorphism $\pi\colon K\to\langle a,b\rangle$, which sends $Q_n\mapsto a$, $Q_{n-2}\mapsto b$, and  all generators $Q_i$, $i\ne n,n-2$ to $1$. Since the curves $c_n$ and $c_{n-2}$ are not disjoint, $a$ and $b$ generate a rank $2$ free group. The image of the commutator $[g,h]$ under $\pi$ is $[\pi(g),\pi(h)]=[a^2b^{-1},ab^{-2}]\ne1$. Therefore, $[g,h]\ne1$ in $K$.

Now, the result of Baudisch~\cite[Theorem~1.2]{Baudi1}, shows that $g$ and $h$ generate a rank $2$ free group. Therefore the action of $\ov{A[D_n]}$ on $\CC$ is non-elementary.

Applying Corollary~\ref{cor:7}, we conclude that $\ov{A[D_n]}$ has property $R_\infty$, and hence, by Lemma~\ref{lem:epi}, $A[D_n]$ has property $R_\infty$ as well.
\end{proof}

\section{Results for large-type Artin groups}\label{sec4}

{\it In this section $\Gamma$ denotes the presentation graph of $A[\Gamma]$, not its Coxeter graph, see the definition below.}

Let $A[\Gamma]$ be a large-type Artin group. Throughout this section we suppose that $\Gamma$ has at least $3$ vertices, ensuring that $A[\Gamma]$ is non-spherical. As is customary when working with general (i.e.\ not spherical and not affine) Artin groups, we consider $\Gamma$ to be the presentation graph of the Artin group (as opposed to its Coxeter graph), which is defined as follows. Let $S$ be a standard generating set for the Artin group. Then the \emph{presentation graph} associated with $S$ is the graph $\Gamma$ whose vertex set $V(\Gamma)$ is $S$, and for which there is an edge with label $m_{st}$ between $s$ and $t$ if and only if $m_{st} \ne \infty$. (I.e.~in the presentation graph, two vertices $s,t$ with $m_{st}=2$ are connected with an edge labeled $2$, and two vertices $s,t$ with $m_{st}=\infty$ are disconnected, whereas in the Coxeter graph the situation is reversed: vertices $s,t$ with $m_{st}=2$ are disconnected, but the ones with $m_{st}=\infty$ are connected with an edge labeled $\infty$.)

We consider the subgroup $\Aut_{\Gamma}(A[\Gamma])$ of $A[\Gamma]$ generated by the following automorphisms:
\begin{itemize}
    \item \emph{inner automorphisms}, i.e.\ conjugations of the form ${\conj_h}\colon g \mapsto h g h^{-1}$;
    \item \emph{graph-induced automorphisms}, i.e.\ automorphisms $\sigma$ that consist of a permutation of the standard generators induced by a label-preserving graph automorphism of $\Gamma$;
    \item \emph{the global inversion}, i.e.\ the automorphism $\chi$ that sends every standard generator to its inverse.
\end{itemize}

Let $A[\Gamma]$ be an Artin group. It is known by \cite{Lek1} that for every full subgraph $\Gamma' \subseteq \Gamma$, the subgroup generated by the vertices of $\Gamma'$ is isomorphic to the Artin group $A[\Gamma']$. Such subgroups are called \emph{standard parabolic subgroups}.

\begin{defin}[\protect{Charney--Davis~\cite{ChaDav1}}]
    Let $A[\Gamma]$ be any Artin group. The \emph{Deligne complex} $D_{\Gamma}$ is the simplicial complex defined as follows:
    \begin{itemize}
        \item The vertex set of $D_{\Gamma}$ is the set of left cosets of the form $g A[\Gamma']$ where $g \in A[\Gamma]$ and $A[\Gamma']$ is a spherical standard parabolic subgroup of $A[\Gamma]$;
        \item For every string of strict inclusion of the form $g_0 A[\Gamma_0] \subsetneq \cdots \subsetneq g_n A[\Gamma_n]$, we put an $n$-simplex between the associated $n+1$ vertices.
    \end{itemize}
\end{defin}    
    The group $A[\Gamma]$ acts on $D_{\Gamma}$ by left-multiplication, and this action is by simplicial isomorphisms.

    It is a classical result for Coxeter groups that if $W[\Gamma]$ is large-type and $\Gamma$ has at least three vertices then $W[\Gamma]$ is infinite. In particular, if $A[\Gamma]$ is large-type then for any induced subgraph $\Gamma' \subseteq \Gamma$ with at least three vertices the subgroup $A[\Gamma']$ is non-spherical. Consequently, a spherical standard parabolic subgroup $A[\Gamma']$ of a large-type Artin group $A[\Gamma]$ necessarily satisfies $|V(\Gamma')| \le 2$. 
    
    In particular, Charney--Davis prove that the associated Deligne complex $D_{\Gamma}$ is a $2$-dimensional simplicial complex which, when endowed with the \emph{Moussong metric}, becomes a CAT(0) metric space, see~\cite[p.\,623]{ChaDav1}.

One can naturally construct an action of $\Inn(A[\Gamma])$ on $D_{\Gamma}$ by declaring that the inner automorphism $\conj_g$ acts as the element $g$. One can actually extend this to a compatible action of $\Aut_{\Gamma}(A[\Gamma])$ on  $D_{\Gamma}$ as follows:

\begin{defin}[Jones--Vaskou~\protect{\cite[Definition~2.13]{JonVas1}}]\label{defi:actionofAut} The group $\Aut_{\Gamma}(A[\Gamma])$ acts by simplicial isomorphisms on $D_{\Gamma}$ as follows. Let $g A[\Gamma']$ be a vertex of $D_{\Gamma}$, where $g \in A[\Gamma]$ and $\Gamma' \subseteq \Gamma$ is an induced subgraph. Then:
\begin{itemize}
    \item (inner automorphisms) ${\conj_h} \cdot g A[\Gamma'] \coloneqq hg A[\Gamma']$;
    \item (graph automorphisms) $\sigma \cdot g A[\Gamma'] \coloneqq \sigma(g) A[\sigma(\Gamma')]$;
    \item (global inversion) $\chi \cdot g A[\Gamma'] \coloneqq \chi(g) A[\Gamma']$.
\end{itemize}
\end{defin}

The following properties of the Deligne complex are known to experts, however they do not seem to be explicitly stated in the literature, so we provide clear statements and proofs here. 

\begin{lem} \label{lem:ActionByIsometries}
    Let $A[\Gamma]$ be a large-type Artin group, and let $f\colon D_{\Gamma} \to D_{\Gamma}$ be a simplicial automorphism of $D_{\Gamma}$. Then:
    \begin{enumerate}
        \item For every vertex $g A[\Gamma']$, the image $f(g A[\Gamma']) = g' A[\Gamma'']$ satisfies $|V(\Gamma')| = |V(\Gamma'')|$;
        \item $f$ is a simplicial isometry, i.e.\ it sends simplices to simplices isometrically;
        \item $f$ is an isometry.
    \end{enumerate}
\end{lem}

\begin{proof}
    (1) Recall that for a simplicial complex $X$, the \emph{link} of a vertex $v$ is the subcomplex $\lk(v) \subseteq X$ obtained as the union of all the simplices that are disjoint from $v$ but which belong to simplices that contain $v$. Also recall that because $A[\Gamma]$ is large-type, any spherical standard parabolic subgroup $A[\Gamma']$ satisfies $|V(\Gamma')| \le 2$. The first statement directly follows from the fact that the link $\lk(g A[\Gamma'])$ is:
    \begin{itemize}
        \item finite, if and only if $|V(\Gamma')| = 0$;
        \item infinite but bounded, if and only if $|V(\Gamma')| = 1$;
        \item unbounded, if and only if $|V(\Gamma')| = 2$,
    \end{itemize}
    where the above follows from ~\cite[Remark 3.4, Proposition E]{Vasko1}. 

    (2) It is enough to prove the second statement for $2$-simplices, as the other types of simplices isometrically embed in them. If $\Delta$ is  such a simplex, we notice that $\Delta$ must have the form $\Delta = \bigl(g_0 A[\Gamma_0], g_1 A[\Gamma_1], g_2 A[\Gamma_2]\bigr)$, where for each $i \in \{0, 1, 2\}$ we have $|V(\Gamma_i)| = i$.  Recall that the Moussong metric $d$ is defined on $\Delta$ by letting $(\Delta, d)$ be isometric to the only Euclidean triangle satisfying:
    \[
    d\bigl(g_0 A[\Gamma_0], g_1 A[\Gamma_1]\bigr) = 1,  \quad
    \angle_{g_1 A[\Gamma_1]} \bigl(g_0 A[\Gamma_0], g_2 A[\Gamma_2]\bigr) = \frac{\pi}{2}, \quad
    \angle_{g_2 A[\Gamma_2]} \bigl(g_0 A[\Gamma_0], g_1 A[\Gamma_1]\bigr) = \frac{\pi}{2m},
    \]
    where $m$ is the coefficient of the dihedral Artin group $A[\Gamma_2]$, see ~\cite{ChaDav1}.
    
    The map $f$ is a simplicial automorphism, so the image $f(\Delta)$ is a $2$-simplex. As for $\Delta$, it takes the form $f(\Delta) = \bigl(g_0' A[\Gamma_0'], g_1' A[\Gamma_1'], g_2' A[\Gamma_2']\bigr)$, where $|V(\Gamma_i')| = i$ for $i \in \{0, 1, 2\}$. By (1), we know that $f(g_i A[\Gamma_i]) = g_i' A[\Gamma_i']$ for $i \in \{0, 1, 2\}$. In particular, by construction of $d$ we have
    \[
    d\bigl(g_0' A[\Gamma_0'], g_1' A[\Gamma_1']\bigr) = 1,  \quad
    \angle_{g_1' A[\Gamma_1']} \bigl(g_0' A[\Gamma_0'], g_2' A[\Gamma_2']\bigr) = \frac{\pi}{2}, \quad
    \angle_{g_2' A[\Gamma_2']} \bigl(g_0' A[\Gamma_0'], g_1' A[\Gamma_1']\bigr) = \frac{\pi}{2m'},
    \]
    where $m'$ is the coefficient of the dihedral Artin group $A[\Gamma_2']$. So all that is left to show is that $m = m'$. For a graph $G$, let us denote by $sys(G)$ the \emph{systole} of $G$, that is the smallest simplicial length of a non-contractible loop in $G$ (also known as the girth of $G$). Since $f$ acts by simplicial automorphisms, it is clear that $sys(\lk(g_2 A[\Gamma_2])) = sys(\lk(g_2' A[\Gamma_2']))$. We will show that $sys(\lk(g_2 A[\Gamma_2])) = 4m$, and that, similarly, $sys(\lk(g_2' A[\Gamma_2'])) = 4m'$. Together, this will imply that $m = m'$.

    Let $V(\Gamma_2)=\{a,b\}$. It was proved in~\cite[Lemma~4.2]{Vasko1} that every non-backtracking loop $\gamma$ in $\lk(g_2 A[\Gamma_2])$ corresponds to a reduced word $w$ in $\{a, b\}$ such that:
    \begin{itemize}
        \item The word $w$ projects to the trivial element in the dihedral Artin group $\langle a, b \rangle$;
        \item Without loss of generality we can write $w = a^{n_1} b^{n_2} \dots x^{n_k}$ for some $x \in \{a, b\}$ and $k \ge 1$;
        \item The simplicial length $\ell(\gamma)$ of $\gamma$ is equal to the integer $2k$. 
    \end{itemize}
    It was also proved in~\cite[Lemma~6]{AppSch1} that we have $k \ge 2m$. Moreover, this lower bound is reached (for instance, take $w$ to be the relator of the dihedral Artin group $\langle a, b \rangle$). It follows that $sys(\lk(g_2 A[\Gamma_2])) = 4m$, as desired.
    
    (3) We now prove the third statement. Let $x$ and $y$ be any two points (not necessarily vertices) in $D_{\Gamma}$. Recall that, by definition, the Moussong metric $d$ can be extended from simplices to the whole of $D_{\Gamma}$ by letting $d(x, y)$ be the infimum of the lengths of the paths connecting $x$ and $y$, where the length of a path $\gamma$ is computed as the sum of the lengths of the sub-paths $\{\gamma_i\}_{i \in I}$ obtained by restricting $\gamma$ to the various simplices $\{\Delta_i\}_{i \in I}$ it travels through. The fact that $(D_{\Gamma}, d)$ is a CAT(0) metric space ensures that there is a unique geodesic $\gamma^{x, y}$ connecting $x$ and $y$ (in particular, the above infimum is always attained). Let us decompose $\gamma^{x, y} = \bigcup\limits_{i \in I} \gamma^{x, y}_i$ as before. The map $f\colon D_{\Gamma} \to D_{\Gamma}$ is a simplicial isometry, so the image $f(\gamma^{x, y}) = \bigcup\limits_{i \in I} f(\gamma^{x, y}_i)$ is a path between $f(x)$ and $f(y)$ for which the length of $f(\gamma^{x, y}_i)$ is the same as the length of $\gamma^{x, y}_i$ for all $i \in I$. It follows that $f(\gamma^{x, y})$ is a path of length $d(x, y)$ between $f(x)$ and $f(y)$, which shows that $d(f(x), f(y)) \le d(x,y)$. Since $f$ is bijective, the same argument works for $f^{-1}$, so that $d(x, y) \le d(f(x), f(y))$. This yields $d(f(x), f(y)) = d(x,y)$ for all $x, y \in D_\Gamma$, i.e.\ $f$ is an isometry. 
\end{proof}

Recall that an Artin group $A[\Gamma]$ is said to be hyperbolic-type, if its Coxeter group $W[\Gamma]$ is Gromov-hyperbolic, and for the large-type Artin groups the requirement to be hyperbolic-type is equivalent to the absence of triangles in $\Gamma$ with all edge labels $3$.

\begin{thm} \label{Thm:SmallAutImpliesRinfty}
Let $A[\Gamma]$ be an Artin group of large and hyperbolic type, and suppose that $\Aut(A[\Gamma]) = \Aut_{\Gamma}(A[\Gamma])$. Then $A[\Gamma]$ has property $R_{\infty}$.
\end{thm}

\begin{proof}
    If $\Gamma$ has $2$ vertices, then $A[\Gamma]$ is a dihedral Artin group, which has property $R_\infty$ by \cite[Theorem~1]{CalSor1}. So we assume that $\Gamma$ has at least $3$ vertices. 
    
    Since $A[\Gamma]$ is of large and hyperbolic type, the Deligne complex $D_{\Gamma}$ (endowed with the Moussong metric) is Gromov-hyperbolic by~\cite[Lemma~5]{Crisp1}. By Lemma \ref{lem:ActionByIsometries} we also know that $\Aut(A[\Gamma])$ acts on $D_{\Gamma}$ by isometries. Also note that the group $A[\Gamma]$ has trivial center. This fact can be distilled from the work of Godelle~\cite{Godel1}; an explicit proof can be found in ~\cite[Corollary~C]{Vasko1}.
    
    Next we are going to show that the action of $A[\Gamma]$ on $D_{\Gamma}$ is non-elementary. In order to do this, we choose two elements $x, y \in A[\Gamma]$ acting elliptically on $D_{\Gamma}$ with disjoint fixed-point sets. In what follows, we denote by $\Fix(g)$ the fixed-point set of an element $g$ acting on $D_{\Gamma}$. We split the argument in two cases.
        
    \underline{Case 1: $\Gamma$ is connected}. There are three standard generators $a, b, c \in V(\Gamma)$ satisfying $m_{ab}, m_{ac} < \infty$, and we pick $x=z_{ab}$ and $y=z_{ac}$, where $z_{st}$ denotes an element generating the center of $\langle s, t \rangle$. It follows from ~\cite[Lemma 8]{Crisp1} that the $\Fix(x)$ is a single vertex: it is the coset that corresponds to the standard parabolic subgroup generated by $a$ and $b$. The same applies to $y$ with the generators $a$ and $c$. In particular, $\Fix(x)$ and $\Fix(y)$ are disjoint.
        
    \underline{Case 2: $\Gamma$ is not connected}. Then there are two standard generators $a, b \in V(\Gamma)$ that lie in distinct connected components of $\Gamma$. It follows from ~\cite[Lemma 8]{Crisp1} that the fixed-point sets $\Fix(a)$ and $\Fix(b)$ are convex trees in $D_{\Gamma}$ (often called \emph{standard trees}). It is a standard result that $\Fix(a)$ and $\Fix(b)$ do not intersect in that case. Indeed, suppose that they do. By~\cite[Corollary~2.18]{HaMaSi1} $\Fix(a) \cap \Fix(b)$ must be a single vertex $v$. We consider the triangle $T$ with vertices $\langle a \rangle$, $\langle b \rangle$ and $v$. Let us denote by $K_{\Gamma}$ the fundamental domain of the action of $A[\Gamma]$ on $D_{\Gamma}$. Note that $K_{\Gamma}$ is convex, so the geodesic between $\langle a \rangle$ and $\langle b \rangle$ lies inside $K_{\Gamma}$. Because $a$ and $b$ lie in different connected components of $\Gamma$, the vertex corresponding to the coset $\{1\}$ disconnects $K_{\Gamma}$, and $\langle a \rangle$ and $\langle b \rangle$ lie in different connected components of $K_{\Gamma} \setminus \{1\}$. Consequently, $\{1\}$ lies on the geodesic from $\langle a \rangle$ to $\langle b \rangle$. It follows that
    \[
    \angle_{\langle a \rangle} \bigl(\langle b \rangle, v\bigr) = \angle_{\langle a \rangle} \bigl(\{1\}, v\bigr).
    \]
    The geodesic from $\langle a \rangle$ to $v$ is contained in the standard tree $\Fix(a)$, the above angle is exactly $\pi/2$, by construction. The same applies to the angle $\angle_{\langle b \rangle} \bigl(\langle a \rangle, v\bigr)$. Because $\Fix(a)$ and $\Fix(b)$ intersect at a single point, it also follows that $\angle_{v} \bigl(\langle a \rangle, \langle b \rangle\bigr) > 0$. Altogether, this shows that the sum of the angles in $T$ is strictly more than $\pi$, which contradicts ~\cite[Chapter II, Exercise 2.12 (1)]{BriHae1}. This contradiction shows that $\Fix(a)$ and $\Fix(b)$ are disjoint, and we can set $x=a$, $y=b$ in Case 2.

    So now we have constructed two elliptic elements $x,y\in A[\Gamma]$ with disjoint fixed-point sets.     By~\cite[Proposition~C]{Marti1}, there is an integer $n > 0$ such that $H=\langle x^n, y^n \rangle$ is a non-abelian free group, and, moreover, an element $g\in H$ acts loxodromically on $D_\Gamma$ if and only if $g$ is not conjugate in $A[\Gamma]$ to a power of $x^n$ or $y^n$. We can exhibit a pair of such elements in $H$, for example, $g=x^ny^n$ and $h=y^nx^n$, by~\cite[Remark~3.3]{Marti1}. Since $g$ and $h$ do not commute, they generate a rank $2$ free subgroup of $H$. 
    In particular, $\langle g,h\rangle$ is not virtually cyclic, hence it cannot be elementary.
This proves that the action of $A[\Gamma]$ on $D_\Gamma$ is non-elementary. 
    
    Now we apply Corollary \ref{cor:7} and conclude that $A[\Gamma]$ has property $R_{\infty}$.
\end{proof}

There has been recent progress on determining the automorphism group of Artin groups, and in particular on finding classes of Artin groups for which the automorphism group $\Aut(A[\Gamma])$ is as small as $\Aut_{\Gamma}(A[\Gamma])$. A first result is the following:

\begin{thm}[\protect{Vaskou~\cite[Theorem A]{Vasko2}}] \label{ThmAutForLTFOI}
    Let $A[\Gamma]$ be a large-type free-of-infinity Artin group of rank at least $3$. Then $\Aut(A[\Gamma]) = \Aut_{\Gamma}(A[\Gamma])$.\qed
\end{thm}

In particular, the known property $R_\infty$ for dihedral Artin groups~\cite[Theorem~1]{CalSor1} and Theorem \ref{Thm:SmallAutImpliesRinfty} yield the following:

\begin{corl}\label{cor:ltfoiht} (=Theorem~\ref{thm:2})
    Let $A[\Gamma]$ be a free-of-infinity Artin group of large and hyperbolic type. Then $A[\Gamma]$ has property $R_{\infty}$.\qed
\end{corl}

Theorem \ref{ThmAutForLTFOI} was recently extended to allow for some graphs that are not necessarily free-of-infinity. We introduce a few relevant definitions.

    An Artin group $A[\Gamma]$ is said to be of \emph{XXXL type} if every coefficient $m_{st}$ is $\ge 6$.
     A graph $\Gamma$ is called \emph{twistless} if it is connected and has no cut-vertex and no separating edge. (We call an edge $e$ between two vertices $s,t$ of a connected graph $\Gamma$ \emph{separating} if the graph $\Gamma\setminus\{s,e,t\}$ obtained from $\Gamma$ by removing the edge $e$ and both vertices $s,t$ incident to it, is disconnected.)

\begin{thm}[\protect{Blufstein--Martin--Vaskou~\cite[Corollary~1.7]{BlMaVa1}}]\label{thm:xxxl}
Let $A[\Gamma]$ be an Artin group of XXXL-type with $\Gamma$ a twistless graph of rank at least $3$. Then $\Aut(A[\Gamma]) = \Aut_{\Gamma}(A[\Gamma])$.\qed
\end{thm}

We immediately obtain the following corollary.

\begin{corl}\label{cor:xxxl}
    (=Theorem~\ref{thm:3}, part 1)
    Let $A[\Gamma]$ be an Artin group of XXXL-type with $\Gamma$ a  twistless graph. Then $A[\Gamma]$ has property $R_{\infty}$.\qed
\end{corl}

We now come to the second extension of Theorem \ref{ThmAutForLTFOI}.

An \emph{admissible decomposition} of a graph $\Gamma$ is a pair of induced subgraphs $\Gamma_1$ and $\Gamma_2$ of $\Gamma$ such that $\Gamma = \Gamma_1 \cup \Gamma_2$ (in the sense that $V(\Gamma)=V(\Gamma_1)\cup V(\Gamma_2)$ and $E(\Gamma)=E(\Gamma_1)\cup E(\Gamma_2)$ for the respective sets of vertices and edges). We say that an admissible decomposition is \emph{twistless} if $\Gamma_1 \cap \Gamma_2$ is not empty, and is not a single vertex nor a single edge. Notice that a graph $\Gamma$ is twistless 
if and only if every admissible decomposition of $\Gamma$ is twistless.
    
    We say that a graph $\Gamma$ is a \emph{twistless star} if $\Gamma$ is the star of a vertex $v \in \Gamma$ and $\Gamma$ is twistless. Note that complete graphs are  examples of twistless stars. 
    
    We say that $\Gamma$ has a \emph{twistless hierarchy terminating in twistless stars} if it is possible to start from $\Gamma$ and perform finitely many successive admissible decompositions until we end up with a collection of graphs that are twistless stars. As examples of such graphs one can take $1$-skeleta of triangulations of surfaces, see~\cite[Remark~5.12]{HuOsVa1}.

In Figure~\ref{fig:twistless} the graph $\Gamma_1$ is a twistless star and the graph $\Gamma=\Gamma_1\cup\Gamma_2$, which is the $1$-skeleton of an octahedron, is an example of a twistless hierarchy terminating in twistless stars, since it is obtained by identifying two copies of $\Gamma_1$ over a $4$-cycle.

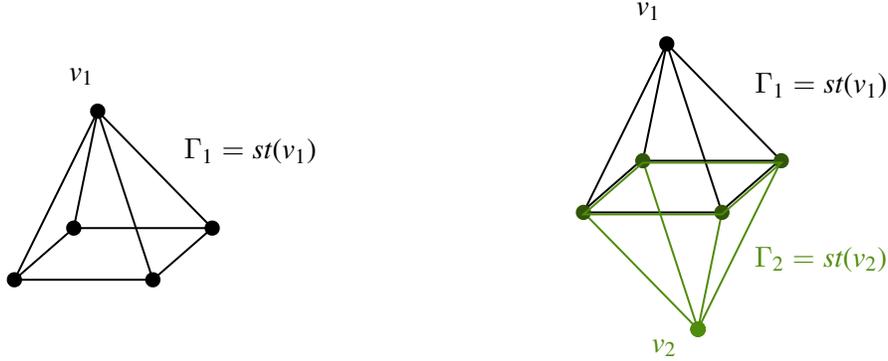
\begin{figure}
\begin{center}
\tikzset{every picture/.style={line width=0.75pt}} 

\begin{tikzpicture}[x=0.75pt,y=0.75pt,yscale=-1,xscale=1,rotate=5]

\draw   (47.93,113) -- (117.76,113) -- (87.83,139.08) -- (18,139.08) -- cycle ;
\draw    (60,54) -- (18,139.08) ;
\draw [shift={(18,139.08)}, rotate = 116.27] [color=black  ][fill=black  ][line width=0.75]      (0, 0) circle [x radius= 3.35, y radius= 3.35]   ;
\draw    (60,54) -- (47.93,113) ;
\draw [shift={(47.93,113)}, rotate = 101.56] [color=black  ][fill=black  ][line width=0.75]      (0, 0) circle [x radius= 3.35, y radius= 3.35]   ;
\draw    (60,54) -- (87.83,139.08) ;
\draw [shift={(87.83,139.08)}, rotate = 71.89] [color=black  ][fill=black  ][line width=0.75]      (0, 0) circle [x radius= 3.35, y radius= 3.35]   ;
\draw [shift={(60,54)}, rotate = 71.89] [color=black  ][fill=black  ][line width=0.75]      (0, 0) circle [x radius= 3.35, y radius= 3.35]   ;
\draw    (60,54) -- (117.76,113) ;
\draw [shift={(117.76,113)}, rotate = 45.61] [color=black  ][fill=black  ][line width=0.75]      (0, 0) circle [x radius= 3.35, y radius= 3.35]   ;
\draw  [color=black  ,draw opacity=1 ] (333.93,79) -- (403.76,79) -- (373.83,105.08) -- (304,105.08) -- cycle ;
\draw    (347,20) -- (305,105.08) ;
\draw    (347,20) -- (334.93,79) ;
\draw    (347,20) -- (374.83,105.08) ;
\draw [shift={(347,20)}, rotate = 71.89] [color=black  ][fill=black  ][line width=0.75]      (0, 0) circle [x radius= 3.35, y radius= 3.35]   ;
\draw    (347,20) -- (404.76,79) ;
\draw [color=blue  ,draw opacity=1 ]   (334.93,79) -- (362.76,164.08) ;
\draw [shift={(362.76,164.08)}, rotate = 71.89] [color=blue  ,draw opacity=1 ][fill=blue  ,fill opacity=1 ][line width=0.75]      (0, 0) circle [x radius= 3.35, y radius= 3.35]   ;
\draw [color=blue  ,draw opacity=1 ]   (305,105.08) -- (362.76,164.08) ;
\draw [shift={(362.76,164.08)}, rotate = 45.61] [color=blue  ,draw opacity=1 ][fill=blue  ,fill opacity=1 ][line width=0.75]      (0, 0) circle [x radius= 3.35, y radius= 3.35]   ;
\draw [color=blue  ,draw opacity=1 ]   (404.76,79) -- (362.76,164.08) ;
\draw [shift={(362.76,164.08)}, rotate = 116.27] [color=blue  ,draw opacity=1 ][fill=blue  ,fill opacity=1 ][line width=0.75]      (0, 0) circle [x radius= 3.35, y radius= 3.35]   ;
\draw [color=blue  ,draw opacity=1 ]   (374.83,105.08) -- (362.76,164.08) ;
\draw [shift={(362.76,164.08)}, rotate = 101.56] [color=blue  ,draw opacity=1 ][fill=blue  ,fill opacity=1 ][line width=0.75]      (0, 0) circle [x radius= 3.35, y radius= 3.35]   ;
\draw [color={rgb, 255:red,0; green,0; blue, 128 }  ,draw opacity=1 ]   (305,105.08) ;
\draw [shift={(305,105.08)}, rotate = 0] [color={rgb, 255:red,0; green,0; blue, 128 }  ,draw opacity=1 ][fill={rgb, 255:red,0; green,0; blue, 128 }  ,fill opacity=1 ][line width=0.75]      (0, 0) circle [x radius= 3.35, y radius= 3.35]   ;
\draw [color={rgb, 255:red,0; green,0; blue, 128 }  ,draw opacity=1 ]   (334.93,79) ;
\draw [shift={(334.93,79)}, rotate = 0] [color={rgb, 255:red,0; green,0; blue, 128 }  ,draw opacity=1 ][fill={rgb, 255:red,0; green,0; blue, 128 }  ,fill opacity=1 ][line width=0.75]      (0, 0) circle [x radius= 3.35, y radius= 3.35]   ;
\draw [color={rgb, 255:red,0; green,0; blue, 128 }  ,draw opacity=1 ]   (404.76,79) ;
\draw [shift={(404.76,79)}, rotate = 0] [color={rgb, 255:red,0; green,0; blue, 128 }  ,draw opacity=1 ][fill={rgb, 255:red,0; green,0; blue, 128 }  ,fill opacity=1 ][line width=0.75]      (0, 0) circle [x radius= 3.35, y radius= 3.35]   ;
\draw [color={rgb, 255:red,0; green,0; blue, 128 }  ,draw opacity=1 ]   (374.83,105.08) ;
\draw [shift={(374.83,105.08)}, rotate = 0] [color={rgb, 255:red,0; green,0; blue, 128 }  ,draw opacity=1 ][fill={rgb, 255:red,0; green,0; blue, 128 }  ,fill opacity=1 ][line width=0.75]      (0, 0) circle [x radius= 3.35, y radius= 3.35]   ;
\draw  [color=blue  ,draw opacity=1 ] (334.93,80) -- (404.76,80) -- (374.83,106.08) -- (305,106.08) -- cycle ;

\draw (101,66.4) node [anchor=north west][inner sep=0.75pt]  [font=\normalsize]  {$\Gamma _{1} =st( v_{1})$};
\draw (388.76,120.94) node [anchor=north west][inner sep=0.75pt]  [color=blue  ,opacity=1 ]  {$\Gamma _{2} =st( v_{2})$};
\draw (43,30.4) node [anchor=north west][inner sep=0.75pt]    {$v_{1}$};
\draw (329,-2.6) node [anchor=north west][inner sep=0.75pt]    {$v_{1}$};
\draw (337,167.4) node [anchor=north west][inner sep=0.75pt]  [color=blue  ,opacity=1 ]  {$v_{2}$};
\draw (389,33.4) node [anchor=north west][inner sep=0.75pt]    {$\Gamma _{1} =st( v_{1})$};

\end{tikzpicture}
\caption{The graph $\Gamma_1$ is a twistless star, and $\Gamma=\Gamma_1\cup\Gamma_2$ is a twistless hierarchy terminating in twistless stars (here $st(v_i)$ denotes the star of vertex $v_i$).
\label{fig:twistless}}
\end{center}
\end{figure}

\begin{thm}[\protect{Huang--Osajda--Vaskou~\cite[Theorem~1.7]{HuOsVa1}}]
    Let $A[\Gamma]$ be a large-type Artin group and suppose that $\Gamma$ admits a twistless hierarchy terminating in twistless stars. Then $\Aut(A[\Gamma]) = \Aut_{\Gamma}(A[\Gamma])$.\qed
\end{thm}

\begin{corl}\label{cor:hierarchy} (=Theorem~\ref{thm:3}, part 2)     Let $A[\Gamma]$ be an Artin group of large and hyperbolic type, and suppose that $\Gamma$ admits a twistless hierarchy terminating in twistless stars. Then $A[\Gamma]$ has property $R_{\infty}$.\qed
\end{corl}

\begin{rem}\label{rem:juh}
In~\cite{Juhas1}, Juh\'asz studied the following subclasses of large-type Artin groups:
\begin{itemize}
\item \emph{extra-large-type} Artin groups, i.e.\ such Artin groups for which $m_{st}\ge4$ for all $s,t\in S$, and
\item the so-called \emph{CLTTF} Artin groups, which were first introduced in~\cite{Crisp1}, and are defined as large-type Artin groups for which the presentation graph $\Gamma$ is connected and has no triangles.
\end{itemize}
For these Artin groups Juh\'asz proved that if $\phi$ is a length-preserving automorphism, then $R(\phi)=\infty$. (An automorphism is called \emph{length-preserving}, if the word length of $\phi(g)$ with respect to a given generating set of a group is equal to the word length of $g$, for each element $g$ of the group.)  In particular, if $\Aut(A[\Gamma])=\Aut_\Gamma(A[\Gamma])$, then each $\phi\in\Out(A[\Gamma])$ has a length-preserving representative in $\Aut(A[\Gamma])$. On the other hand, Lemma~2.1 of~\cite{FeGoDa1} implies that if $R(\phi_\alpha)=\infty$ for all representatives $\{\phi_\alpha\}_{\alpha\in\Out(G)}\subseteq\Aut(G)$ of $\Out(G)$, then $G$ has property $R_\infty$. Combining this with our knowledge of automorphism groups given in Theorems~\ref{ThmAutForLTFOI} and \ref{thm:xxxl} and the above-cited results of Juh\'asz, we get another proof that the following classes of Artin groups have property $R_\infty$:
\begin{enumerate}
\item extra-large-type free-of-infinity Artin groups;
\item XXXL-type Artin groups with $\Gamma$ twistless. 
\end{enumerate}
Notice that these classes are subsumed in Corollaries~\ref{cor:ltfoiht} and \ref{cor:xxxl}. Remarkably, Juh\'asz was using techniques of small cancellation, which capture the non-positive geometry of the presentation complex of a group. Our Theorem~\ref{Thm:SmallAutImpliesRinfty} uses action of groups on  Gromov-hyperbolic spaces instead.
\end{rem}

\section{Proof of Delzant's Lemma}\label{app}

Delzant's Lemma appeared in~\cite[Lemma\,3.4]{LevLus1} with a brief proof sketch, which is repeated (essentially verbatim) in several subsequent papers on the 
$R_\infty$ property, see~\cite{Felsh1,FeGoDa1,CalSor1}. We provide a detailed proof of it here. While our construction differs from that of~\cite{LevLus1}, the underlying idea, namely, using the stable translation length to produce infinitely many conjugacy classes, is the same. We thank R\'emi Coulon, Mark Hagen and Gilbert Levitt for useful discussions on this matter. 

We note that an unpublished preprint of Taback and Whyte~\cite{TabWhy1} outlines a different approach to establishing the existence of infinitely many distinct conjugacy classes in a suitable coset of isometries of $X$. They use it to deduce the $R_\infty$ property for groups acting characteristically on a Gromov-hyperbolic space.

\begin{delzantlemB}
Let $\Gamma$ be a group acting non-elementarily by isometries on a geodesic Gromov-hyperbolic space (which is not assumed to be proper), and let $K$ be a normal subgroup of\/ $\Gamma$ such that the quotient $\Gamma/K$ is abelian. Then every coset of $K$ contains infinitely many conjugacy classes.  
\end{delzantlemB}

We recall the definition and the properties of the Gromov product and the Gromov boundary $\partial X$. For additional information the reader is referred to~\cite{Coulo1} or \cite{Vaisa1}.

Throughout this section we let $(X,d)$ be a geodesic Gromov $\delta$-hyperbolic space, not necessarily proper. Let $o\in X$ be a fixed origin. 
For arbitrary $x,y\in X$, the \emph{Gromov product} of $x$ and $y$ at $o\in X$ is
\[
(x,y)_o:=\frac12\left(d(x,o)+d(y,o)-d(x,y)\right).
\]

A sequence of points $(x_n)$ of $X$ is called a \emph{Gromov sequence}, if $(x_n,x_m)_o\to\infty$ as $n\to\infty$ and $m\to\infty$. (I.e.\ for every $R>0$, there is $N>0$ such that for all $n,m\ge N$, we have $(x_n,x_m)_o\ge R$.) This notion does not depend on the choice of the origin $o\in X$, since for any other basepoint $o'\in X$ we have
$\left|(x_n,x_m)_o-(x_n,x_m)_{o'}\right|\le d(o,o')$, by~\cite[Lemma\,2.8(4)]{Vaisa1}.

Two Gromov sequences $(x_n)$ and $(y_n)$ are called \emph{equivalent}, if $(x_n,y_n)_o\to\infty$ as $n\to\infty$. This is indeed an equivalence relation: symmetry and reflexivity are obvious, and transitivity follows from the property
$(x,z)_o\ge\min\left((x,y)_o,(y,z)_o\right)-\delta$, for any $x,y,z\in X$, see~\cite[Prop.\,III.H.1.22]{BriHae1}. 

The \emph{Gromov boundary $\partial X$} of $X$ is defined to be the set of equivalence classes of Gromov sequences in $X$. If $\xi$ is the equivalence class of a Gromov sequence $(x_n)$, we say that $(x_n)$ \emph{converges to $\xi$}. 
\begin{lem}\label{lem:bdrygp}
Let $(x_n)$, $(y_n)$ be Gromov sequences in $X$, and let $\xi,\zeta\in\partial X$ be their respective equivalence classes. Then: $\xi\ne\zeta$ if and only if the set
$\left\{\,(x_i,y_j)_o\,\left|\vphantom{(x_i,y_j)}\right.\, i,j\in\N\,\right\}$
is bounded from above.
\end{lem}
\begin{proof}
Assume $\left\{\,(x_i,y_j)_o\,\left|\vphantom{(x_i,y_j)}\right.\, i,j\in\N\,\right\}$ is unbounded. Then for each $k\in\N$ there exist indices $i_k$, $j_k$ such that $\left(x_{i_k},y_{j_k}\right)_o\ge k$.
We claim that in this case $i_k\to\infty$ as $k\to\infty$. If not, then some index $i_0$ occurs infinitely often among the $i_k$. Pass to a subsequence (which we still denote by $k$). Then we will have for all these $k$, $\left(x_{i_0},y_{j_k}\right)_o\ge k$. But 
$\left(x_{i_0},y_{j_k}\right)_o\le d\left(x_{i_0},o\right)$ (see e.g.~\cite[Lemma\,2.8(3)]{Vaisa1}), which yields a contradiction. Thus $i_k\to\infty$ and, by symmetry, $j_k\to \infty$ as $k\to \infty$. Therefore we may assume that both $(i_k)$ and $(j_k)$ are strictly increasing. Thus, the inequality $\left(x_{i_k},y_{j_k}\right)_o\ge k$ implies that subsequences $(x_{i_k})$ and $(y_{j_k})$ are equivalent. By~\cite[Lemma\,5.3(1)]{Vaisa1}, any Gromov sequence is equivalent to each of its subsequences, which implies that $(x_n)$ and $(y_n)$ are equivalent, and $\xi=\zeta$.

Now assume that $\xi=\zeta$. By definition, this means that $(x_n,y_n)_o\to\infty$, which makes the set\\ $\left\{\,(x_i,y_j)_o\,\left|\vphantom{(x_i,y_j)}\right.\, i,j\in\N\,\right\}$ unbounded.
\end{proof}

We will need an estimate of the length between the endpoints of a polygonal path in terms of Gromov products of its intermediate points.

\begin{lem}[\protect{\cite[Lemma\,3.6]{GiMiOs1}}]\label{lem:gmo}
Let $p_0,\dots,p_r\in X$. Suppose there exists $\kappa>0$ such that, for every $1\le i\le r-1$,
\[
d(p_{i-1},p_{i+1})\ge \kappa+2\delta+\max\left\{d(p_{i-1},p_{i}),d(p_{i},p_{i+1})\right\}.
\]
Then $d(p_0,p_r)\ge \kappa r$.
\pushQED{\qed}\qedhere\popQED
\end{lem}

\begin{corl}\label{cor:gmo}
Let $p_0,\dots,p_r\in X$ and suppose that there exists $C>0$ such that for all $1\le i\le r-1$, 
\[
\left(p_{i-1},p_{i+1}\right)_{p_i}\le C.
\]
Let $L=\min\limits_{1\le i\le r}d\left(p_{i-1},p_i\right)$ and suppose that $L>2C+2\delta$. Then for $\kappa=L-(2C+2\delta)>0$ we have
\[
d(p_0,p_r)\ge\kappa r.
\]
\end{corl}
\begin{proof}
From the definition of Gromov product, we have
\[
d\left(p_{i-1},p_{i+1}\right)=d\left(p_{i-1},p_{i}\right)+d\left(p_{i},p_{i+1}\right)-2\left(p_{i-1},p_{i+1}\right)_{p_i}.
\]
Denote
\[
M=\max\left(d\left(p_{i-1},p_{i}\right),d\left(p_{i},p_{i+1}\right)\right),\qquad
m=\min\left(d\left(p_{i-1},p_{i}\right),d\left(p_{i},p_{i+1}\right)\right).
\]
Then 
\[
d\left(p_{i-1},p_{i+1}\right)=M+m-2\left(p_{i-1},p_{i+1}\right)_{p_i}\ge M+m-2C\ge M+L-2C=M+\kappa+2\delta, 
\]
since $m\ge L$ and $L-2C=\kappa+2\delta$. Now we can apply Lemma~\ref{lem:gmo} and conclude that $d(p_0,p_r)\ge \kappa r$.
\end{proof}

Recall that an isometry $g$ of $X$ is called \emph{loxodromic}, if the map from $\Z$ to $X$ which sends $m$ to $g^mx$ is a quasi-isometric embedding for some (and hence any) point $x\in X$. If $g$ is loxodromic, then it defines two points on the boundary $\partial X$, which are denoted $g^{\infty}$ and $g^{-\infty}$, namely, the equivalence classes of Gromov sequences $(g^nx)$ and $(g^{-n}x)$, $n\in\N$, respectively. The fact that these sequences are Gromov and that the resulting boundary points are independent of the choice of 
$x\in X$ follows, for instance, from~\cite[Proposition\,2.5\,(ii)]{Coulo1} and \cite[Lemma\,5.3(6)]{Vaisa1}.

Remarkably, the condition of Corollary~\ref{cor:gmo} is satisfied for orbits of independent loxodromic elements, as the next lemma shows.
\begin{lem}\label{lem:loxgpbound}
Let $a$, $b$ be loxodromic isometries of $X$ with $\left\{a^\infty, a^{-\infty}\right\}\cap\left\{b^\infty,b^{-\infty}\right\}=\varnothing$. Fix $o\in X$. Then there exists $C=C(a,b,o)>0$ such that for all $m,n\ge1$ and all choices of signs, 
\[
(a^{\pm m}o, b^{\pm n}o)_o\le C, \quad
(a^{m}o,\; a^{-n}o)_o \le C,
\quad\text{and}\quad
(b^{m}o,\; b^{-n}o)_o \le C.
\]
\end{lem}

\begin{proof}
First consider Gromov sequences for fixed $ \epsilon,\epsilon'\in\{\pm1\}$:
\[
x_m=a^{\epsilon m}o,\qquad y_n=b^{\epsilon' n}o, \qquad \text{with\quad$ m,n\ge1$}.
\]
Since $a$ is loxodromic, $a^{\epsilon m}o$ converges to $a^{\epsilon\,\infty}$. Similarly, $b^{\epsilon'n}o$ converges to $b^{\epsilon'\,\infty}$. Since $\left\{a^\infty, a^{-\infty}\right\}\cap\left\{b^\infty,b^{-\infty}\right\}=\varnothing$, we have $a^{\epsilon\,\infty}\ne b^{\epsilon'\,\infty}$. By Lemma~\ref{lem:bdrygp}, for each choice of $\epsilon,\epsilon'\in\{\pm1\}$, there exist a constant $C_{\epsilon\epsilon'}>0$ such that $\bigl(a^{\epsilon m}o,b^{\epsilon' n}o\bigr)_o\le C_{\epsilon\epsilon'}$ for all $m,n\in\N$. Then for $C_0=\max_{\epsilon,\epsilon'\in\{\pm1\}}\left(C_{\epsilon\epsilon'}\right)$, we will have the required inequality $\left(a^{\pm m}o, b^{\pm n}o\right)_o\le C_0$.

Now consider Gromov sequences
\[
x_m=a^mo\to a^\infty,\qquad y_n=a^{-n}o\to a^{-\infty},\qquad m,n\ge1.
\]
Since $a$ is loxodromic, $a^\infty\ne a^{-\infty}$. Therefore, by Lemma~\ref{lem:bdrygp}, there is a constant $C_a>0$ such that $(a^{m}o,\; a^{-n}o)_o \le C_a$ for all $m,n\ge1$. 

Similarly, there exist $C_b>0$ such that $(b^{m}o,\,b^{-n}o)_o \le C_b$ for all $m,n\ge1$.

Taking $C=\max(C_0,C_a,C_b)$ finishes the proof.
\end{proof}

\begin{rem}
    Notice that the set $\{\,(a^no,a^mo)_o\mid m,n\ge1\}$ is unbounded, since for any loxodromic element $a$, the sequence $(a^no)$ is Gromov.
\end{rem}

We will also need a useful property of the stable translation length, which we now recall. 

Let $o\in X$ be fixed. For an isometry $g$ of $X$, the \emph{stable translation length} is 
\[
\tau(g)=\lim_{n\to\infty}\frac {d(o,g^n o)}{n}.
\]
It is known that $\tau(g)$ always exists and is independent of $o\in X$, see e.g.~\cite[Ex.II.6.6(1)]{BriHae1}. Since $d(o, (hgh^{-1})^no)=d(h^{-1}o,g^nh^{-1}o)$, it follows that $\tau(hgh^{-1})=\tau(g)$ for all isometries $g,h$ of $X$.

\begin{lem}\label{lem:stl}
For any isometry $g$ and any point $x\in X$, we have:
\[
d(x,g^n x)\ge |n|\,\tau(g), \quad\text{for all\ \ }n\in\Z.
\]
\end{lem}
\begin{proof}
Denote $f(n)=d(x,g^nx)$, and assume first that $n\ge 1$. Then, for any $k\ge1$ we have:
\[
f(k\,n)=d\left(x,g^{k\,n}x\right)\le\sum_{i=0}^{k-1}d\left(g^{i\,n}x,g^{(i+1)n}x\right)=\sum_{i=0}^{k-1}d\left(x,g^nx\right)=k\,f(n).
\]
Dividing both sides by $k\,n$, we get:
\[
\frac{f(k\,n)}{k\,n}\le \frac{f(n)}{n}.
\]
Taking $k\to\infty$, we notice that $k\,n\to\infty$ as well, and, by the definition of $\tau(g)$, we get $\frac{f(k\,n)}{k\,n}\to \tau(g)$. Since the right-hand side of the above inequality is independent of $k$, we conclude that $\tau(g)\le\frac{f(n)}{n}$, and hence $f(n)\ge n\,\tau(g)$.

Now, if $n<0$, then, using the fact that $g^n$ is an isometry and that $d$ is symmetric, we get: $f(n)=d(x,g^nx)=d(g^{-n}x,x)=d(x,g^{-n}x)=f(-n)$, which implies that $f(n)\ge |n|\,\tau(g)$ for all $n\in\Z$.
\end{proof}

Now we are ready to prove that sufficiently large powers of independent loxodromic elements generate a quasi-isometrically embedded nonabelian free group. Our proof is modeled on~\cite[Lemma~3.2]{TayTio1}, using the auxiliary results established above.

\begin{prop}\label{prop:qi-emb}
Let $a,b\in\Gamma$ be independent loxodromic elements, i.e.\ such that 
\[\left\{a^\infty,a^{-\infty}\right\}\cap\left\{b^\infty,b^{-\infty}\right\}=\varnothing.
\]
Then, for some $N\ge1$, the elements $A=a^N$, $B=b^N$
freely generate the rank $2$ free group $\langle A,B\rangle\simeq F_2$. Moreover, there exists $\kappa>0$ such that for any reduced word $W$ over $\left\{A^{\pm1}, B^{\pm1}\right\}$ with length $|W|$, we have 
\[
d\bigl(o,W\,o\bigr)\ge\kappa\, |W|.
\]
Therefore, the orbit map $\Phi\colon\langle A,B\rangle\to X$, $g\mapsto g\,o$, is a quasi-isometric (indeed, a bi-Lipschitz) embedding.
\end{prop}

\begin{proof}
Let $o\in X$ be fixed, and for the given independent loxodromic $a$ and $b$ let $C$ be the constant given by Lemma~\ref{lem:loxgpbound}. Let $N\in \N$ be arbitrary and let $A=a^N$, $B=b^N$. For a reduced word $W=s_1\dots s_n$ with $s_i\in\left\{A^{\pm1},\,
B^{\pm1}\right\}$, let $p_0=o$ and $p_i=s_1\dots s_i\,o$, for $i=1,\dots,n$. Notice that $s_{i+1}\ne s_i^{-1}$ for all $i$, so the pair $(s_i^{-1}, s_{i+1})$ is either $(A^{\pm1},B^{\pm1})$,  or $(A,A^{-1})$, $(A^{-1},A)$, $(B,B^{-1})$, $(B^{-1},B)$, and the constant $C$ from Lemma~\ref{lem:loxgpbound} bounds $(s_i^{-1}o,s_{i+1}o)_o$ in all cases. Hence we have:
\begin{multline*}
\left(p_{i-1},p_{i+1}\right)_{p_i}=
\frac12\left(d(p_{i-1},p_i)+d(p_{i+1},p_i)-d(p_{i-1},p_{i+1})\right)=\\ \frac12\left(d(s_i^{-1}o,o)+d(s_{i+1}o,o)-d(s_i^{-1}o,s_{i+1}o)\right)=\left(s_i^{-1}o,s_{i+1}o\right)_o\le C.
\end{multline*}

By Lemma~\ref{lem:stl}, $d(o,A^{\pm 1}o)\ge |N|\,\tau(a)$ and $d(o,B^{\pm 1}o)\ge |N|\,\tau(b)$. Since $a$ and $b$ are loxodromic, $\tau(a)$ and $\tau(b)$ are positive, and we can choose $N$ large enough so that \[
L=\min\left\{d\left(o,A^{\pm 1}o\right),\,d\left(o,B^{\pm 1}o\right)\right\}> 2C+2\delta.
\]
Now we can apply Corollary~\ref{cor:gmo} with $\kappa = L-(2C+2\delta)>0$ and conclude that $d(o,W\,o)\ge n\,\kappa$. 

This implies that the group generated by $\langle A,B\rangle$ is free of rank $2$, since a nontrivial reduced word $W$ in $A^{\pm1},B^{\pm1}$ cannot move $o$ by a positive distance and be the identity.

To prove that the orbit map $\Phi\colon\langle A,B\rangle\to X$, $g\mapsto g\,o$, is a quasi-isometric embedding, we need to show that there exist constants $\lambda>0$ and $\varepsilon\ge0$ such that 
\[
\frac{1}{\lambda} 
\,d_{Cay}(g,h)-\varepsilon\le d\left(\Phi(g),\Phi(h)\right)\le \lambda\, d_{Cay}(g,h)+\varepsilon,
\]
for all $g,h\in \langle A,B\rangle$, where $d_{Cay}$ is the path metric in the corresponding Cayley graph of $\langle A,B\rangle$.
Let $|\cdot|$ denote the word length of $W$ in the generators $\left\{A^{\pm1},B^{\pm1}\right\}$. The above double inequality is equivalent to 
\[
\frac{1}{\lambda} 
\,|W|-\varepsilon\le d\left(\vphantom{W'}o,W\,o\right)\le \lambda\, |W|+\varepsilon,
\]
where $W=g^{-1}h$ is an arbitrary word in $\langle A,B\rangle$. Notice that we already proved that $d(o,W\,o)\ge \kappa|W|$, so only the upper bound needs to be shown.

Denote $M=\max\left\{d(o,A\,o),\,d(o,A^{-1}o),\,d(o,B\,o),\,d(o,B^{-1}o)\right\}$. Then (with the above notation for $p_i=s_1\dots s_i\,o$):
\[
d(\vphantom{W'}o,W\,o)=d\left(p_0,p_n\right)\le\sum_{i=0}^{n-1}d\left(p_i,p_{i+1}\right)=
\sum_{i=0}^{n-1}d\left(o,s_{i+1}o\right)\le M\cdot n=M\cdot|W|.
\]
This shows that for all $W$, we have:
\[
\kappa\cdot |W|\le d(o,W\,o)\le M\cdot |W|,
\]
which establishes a quasi-isometric embedding with $\lambda=\max(M,1/\kappa)$ and $\varepsilon=0$, i.e.\ a bi-Lipschitz embedding.
\end{proof}

In what follows we identify the free group $H$ with its Cayley graph, which is a Gromov-hyperbolic metric space (a tree).

\begin{corl}\label{cor:phi}
For $H=\langle A,B\rangle$ as in Proposition~\ref{prop:qi-emb}, the orbit map\/ $\Phi\colon H\to X$, $g\mapsto g\,o$, induces an injective continuous map between the boundaries\/ $\Phi^*\colon\partial H\hookrightarrow\partial X$. Moreover, $\Phi^*$ is a homeomorphism of $\partial H$ onto its image\/ $\Phi^*(\partial H)\subset \partial X$.
\end{corl}
\begin{proof}
We refer the reader to \cite[Section\,5]{Vaisa1} for the definition and  properties of the topology on $\partial X$ (for $X$ not necessarily proper). In particular, in Section~5.32 V\"ais\"al\"a proves that any quasi-isometric embedding between Gromov-hyperbolic length spaces $f\colon X\to Y$ induces an injective continuous map between the boun\-da\-ri\-es $f^\ast\colon\partial X\hookrightarrow \partial Y$. (Note that  V\"ais\"al\"a calls length spaces `intrinsic' and quasi-isometric embeddings `quasi-isometries'.)
Applying this to the quasi-isometric embedding $\Phi\colon H\to X$ yields an injective
continuous map $\Phi^*\colon\partial H\hookrightarrow\partial X$.

It remains to show that $\Phi^*$ is a homeomorphism onto its image.
Since $\partial H$ is homeomorphic to a Cantor set, it is compact; and since $\partial X$ is metrizable (see~\cite[Proposition\,5.31]{Vaisa1}), it is  Hausdorff. 
Since a continuous bijection from a compact space to a Hausdorff space is a homeomorphism, we conclude that $\Phi^*$ is a homeomorphism of $\partial H$ onto its image $\Phi^*(\partial H)\subset \partial X$.
\end{proof}

To make the topology on $\partial H$
concrete, we describe it using geodesic rays in the Cayley graph of the free group, treating them as infinite freely reduced words in the generators of $H$.

Recall that for any proper geodesic Gromov-hyperbolic space $Y$ one can define $\partial Y$ as the set of equivalence classes of geodesic rays $c\colon[0,\infty)\to Y$, with two such rays $c,c'$ being equivalent if and only if $\sup_td(c(t),c'(t))$ is finite, see~\cite[III.H.3]{BriHae1}. Under this definition, a basic neighborhood of a point of $\partial Y$ represented by a geodesic ray $c_0\colon[0,\infty)\to Y$ with $c_0(0)=o$ is defined as
\[
V_n(c_0)=\{\text{\,geodesic rays $c\colon[0,\infty)\to Y$ with $c(0)=o$ and $d(c_0(n),c(n))<k$\,}\},
\]
where $k>2\delta$ is a fixed constant and $n\ge 1$ is arbitrary, see~\cite[Lemma\,III.H.3.6]{BriHae1}. 

For $H=\langle A,B\rangle$ a free group, we identify $Y$ with the corresponding Cayley graph on generators $A,B$, and these constructions simplify significantly. First, the condition $\sup_td(c(t),c'(t))<\infty$, which establishes the equivalence of geodesic rays, means that the rays $c,c'$ coincide for all $t\ge t_0$ for big enough $t_0$. Thus one can pick a representative of $\partial Y$ as a geodesic ray starting at $1$, or, equivalently, as an infinite freely reduced word in $\{\, A^{\pm1},B^{\pm1}\,\}$.

Since the Cayley graph for $H$ is $\delta$-hyperbolic with $\delta=0$, we can choose $k>2\delta$ to be in between $0<k<1$, and then the condition $d(c_0(n),c(n))<k$ in the definition of neighborhoods $V_n(c_0)$ means that $c_0(n)=c(n)$, which implies that rays $c_0$ and $c$ share the initial segment of length $n$. Hence, we can identify neighborhoods $V_n(c_0)$ with the sets (called `cylinder neighborhoods', see e.g.~\cite{Kapov1}):
\[
\Cyl(v)=\left\{\,\text{infinite reduced words in $\{\,A^{\pm1},B^{\pm1}\,\}$ whose initial segment is the word $v$}\,\right\}.
\]
Since, obviously, $\Cyl(1)=\partial H$, and for every $v\in H$,
\[
\Cyl(v)=\bigsqcup_{\left\{x\in \{A^{\pm1},B^{\pm1}\}\,:\,|vx|=|v|+1\right\}}\Cyl(vx),
\]
we see that every set $\Cyl(v)$ is open and closed, and these sets comprise a basis of topology on $\partial H$.

Notice that by~\cite[Proposition\,5.35]{Vaisa1} any isometry $u\in\Isom(X)$ induces a homeomorphism $u^*$ of $\partial X$. To ease notation, in what follows we denote $u^*$ with the same letter $u$.

\begin{lem}\label{lem:uc}
Let $X$ and $\Gamma$ be as before and let $H=\langle A,B\rangle$ be a rank $2$ free subgroup of $\Gamma$ such that its orbit map $\Phi\colon H\to X$, $g\mapsto go$ is a quasi-isometric embedding. Then for any $u\in\Isom(X)$ there exists a loxodromic element $c\in H$ such that $u\,c^{\infty}\ne c^{-\infty}$.
\end{lem}

\begin{proof}
Since $\Phi\colon H\to X$ is a quasi-isometric embedding, each infinite order element $g\in H$ is loxodromic on $X$ (since the map $\mathbb Z\to X$, $n\mapsto g^n o$, is a quasi-isometric embedding itself). 
By Corollary~\ref{cor:phi}, $\Phi^*$ induces a homeomorphism of $\partial H$ onto its image in $\partial X$, hence we may identify $\partial H$ with the corresponding subspace $\Phi^*(\partial H)\subset \partial X$.

Assume now for the sake of contradiction, that 
$u(c^\infty)=c^{-\infty}$ for all loxodromic $c\in\langle A,B\rangle$. In particular, $u(A^\infty)=A^{-\infty}$.

For each $n\ge1$, denote $c_n\coloneqq A^nB$. Then $c_n$ is loxodromic and its attracting point $c_n^\infty\in\partial H$ is represented by the infinite word
\[
c_n^\infty = (A^nB)^\infty = A^n B\, A^n B\,A^n B\dots
\]
Fix $k\ge1$. Notice that for all $n\ge k$, $c_n^\infty$ begins with $A^n$, hence, in particular, with $A^k$. This means that for any fixed $k$ we have
\[
n\ge k\,\Longrightarrow \,c_n^\infty\in \Cyl(A^k).
\]
Since the sets $\Cyl(A^k)$, $k\ge 1$, form a basis of neighborhoods at $A^\infty$, we conclude that $c_n^\infty\to A^\infty$ as $n\to\infty$. Since $u$ acts continuously on $\partial X$, we see that
$u(c_n^\infty)\to u(A^\infty)$,
which implies, due to our assumption $u(c^\infty)=c^{-\infty}$ for all loxodromic $c$, that
\[
c_n^{-\infty}=u(c_n^\infty)\to u(A^{\infty})=A^{-\infty}, \quad\text{as $n\to \infty$.}
\]
On the other hand, $c_n^{-\infty}=(c_n^{-1})^{\infty}$ and $c_n^{-1}=B^{-1}A^{-n}$, so that
\[
c_n^{-\infty} = B^{-1}A^{-n}B^{-1}A^{-n}B^{-1}A^{-n}\ldots\,\in \Cyl(B^{-1}), \quad\text{for all $n\ge1$.}
\]
But $A^{-\infty}\in \Cyl(A^{-1})$, and $\Cyl(A^{-1})\cap\Cyl(B^{-1})=\varnothing$. Since $\Cyl(B^{-1})$ is closed, a sequence contained in $\Cyl(B^{-1})$ cannot converge to a point outside of it. This contradicts the established claim that $c_n^{-\infty}\to A^{-\infty}$, thus yielding a contradiction. Hence there exists a loxodromic $c\in \langle A,B\rangle$ such that $u\,c^\infty\ne c^{-\infty}$.
\end{proof}

\begin{prop}\label{prop:lin-tau}
Let $u\in \Isom(X)$ and let $c\in \Isom(X)$ be loxodromic such that
$u\bigl(c^{\infty}\bigr)\neq c^{-\infty}$.
Then there exist $n_0\in\N$ and a constant $C\ge 0$ such that for all $n\ge n_0$, for the stable translation length of the elements $uc^n$ we will have the estimate:
\[
\tau(u c^n) \ge n\,\tau(c) - C.
\]
In particular, $\tau(uc^n)\to\infty$ as $n\to\infty$.
\end{prop}

\begin{proof} For $n\ge 1$ denote $g_n=u c^n$. Fix $o\in X$ and consider the sequences
\[
x_n=g_n o= u c^n o,
\qquad
y_n=g_n^{-1}o=c^{-n}u^{-1}o.
\]
Since $c$ is loxodromic, $(c^n o)$ and $(c^{-n}u^{-1}o)$ are Gromov sequences, representing $c^{\infty}$ and $c^{-\infty}$, respectively; hence $(x_n)$ and $(y_n)$ are Gromov sequences representing $u(c^{\infty})$ and $c^{-\infty}$.
Since by the assumption $u(c^{\infty})\neq c^{-\infty}$, Lemma~\ref{lem:bdrygp} implies that there exists $C_0>0$ such that
\[
(y_n,x_n)_o=\bigl(g_n^{-1}o,g_n o\bigr)_o\le C_0
\qquad\text{for all }n\ge 1.
\]
Fix such an $n$ and let $r\ge 1$. Define points
\[
p_i:=g_n^i\,o \qquad \text{for\,\, }i=0,1,\dots,r.
\]
For each $i=1,\dots,r-1$, a similar computation as in the proof of Proposition~\ref{prop:qi-emb} gives: 
\[
\left(p_{i-1},p_{i+1}\right)_{p_i}
=
\bigl(g_n^{-1}o,g_n o\bigr)_o
\le C_0.
\]
Moreover, each segment $[p_{i-1},p_i]$ has the same length:
\[
d(p_{i-1},p_i)=d\left(g_n^{i-1}o,g_n^i o\right)=d\!\left(o,g_n o\right),
\]
Therefore the quantity $L=\min_{1\le i\le r} d\left(p_{i-1},p_i\right)$ appearing in Corollary~\ref{cor:gmo} is equal to
\[
L=d(o,g_n o)=d(o,uc^n o).
\]
We claim that for every $n\ge 1$,
\begin{equation}\tag{*} \label{eq:ucn-lower}
d(o,g_n o)=d(o,uc^n o) \ge n\,\tau(c)-d(o,uo).
\end{equation}
Indeed, by the triangle inequality,
\[
d(o,uc^n o)=d(u^{-1}o,c^n o)\ge d(o,c^n o)-d(o,u^{-1}o)=d(o,c^n o)-d(o,uo),
\]
and by Lemma~\ref{lem:stl} we have $d(o,c^n o)\ge n\,\tau(c)$, which shows~\eqref{eq:ucn-lower}.

Since $\tau(c)>0$, we can choose $n_0$ so large that $n\,\tau(c)-d(o,uo)>2C_0+2\delta$ for all $n\ge n_0$.
Then~\eqref{eq:ucn-lower} implies that for all $n\ge n_0$,
\[
L=d(o,g_n o)>2C_0+2\delta.
\]
Therefore Corollary~\ref{cor:gmo} applied to $\left(p_i\right)_{i=0}^r$ yields for all $r\ge 1$:
\[
d(o,g_n^r o)=d(p_0,p_r) \ge \kappa\,r,
\]
where $\kappa=d(o,g_n o)-(2C_0+2\delta)>0$.
Dividing the above inequality by $r$ and letting $r\to\infty$ gives
\[
\tau(g_n)\ge\kappa= d(o,g_n o)-(2C_0+2\delta).
\]
Combining with~\eqref{eq:ucn-lower} we obtain that, for all $n\ge n_0$,
\[
\tau(g_n)\ge n\,\tau(c) -\bigl(d(o,uo)+2C_0+2\delta\bigr).
\]
This proves the Proposition if we take $C=d(o,uo)+2C_0+2\delta$, which is independent of $n$.
\end{proof}

Now we are ready to prove Delzant's Lemma.
\begin{proof}[Proof of Delzant's Lemma]
Let $X$ be a geodesic Gromov-hyperbolic space on which $\Gamma$ acts non-ele\-men\-ta\-ri\-ly. Fix a basepoint $o\in X$ and let $a,b\in\Gamma$ be independent loxodromics. Then by Proposition~\ref{prop:qi-emb}, there exist $N\ge1$ such that $A=a^N$, $B=b^N$ generate a rank $2$ free group $H=\langle A,B\rangle$, whose orbit map $\Phi\colon H\to X$, $g\mapsto g\,o$ is a quasi-isometric embedding. Since $\Gamma/K$ is abelian, $[\Gamma,\Gamma]\subseteq K$, and, in particular, $[H,H]\subseteq K$. Since $[H,H]$ is free, we can choose a rank $2$ subgroup $H'\le[H,H]$. Since $H'$ is a finitely generated subgroup of the free group $H$, it is quasi-convex in $H$ and hence quasi-isometrically embedded into $H$. Composing the inclusion $H'\hookrightarrow H$ with the quasi-isometric embedding $\Phi\colon H\to X$, we conclude that $H'$ is quasi-isometrically embedded into $X$ as well, and, moreover, $H'\le K$.

Fix any coset $uK\in \Gamma/K$. By Lemma~\ref{lem:uc}, there exists a loxodromic element $c\in H'$ such that $u\,c^\infty\ne c^{-\infty}$. By Proposition~\ref{prop:lin-tau}, there exist $n_0\in\N$ and $C\ge0$ such that for the stable translation length of the element $uc^n\in uK$ we have
\[
\tau(uc^n)\ge n\,\tau(c)-C,
\]
for all $n\ge n_0$. In particular, $\tau$ attains infinitely many values on the set $\{\, uc^n\mid n\in\N\,\}$. Since $\tau$ is invariant under conjugation, we conclude that the coset $uK$ meets infinitely many conjugacy classes of~$\Gamma$. 

Let $g\in uK$. Then $g=uk$ for some $k\in K$, and for an arbitrary $h\in\Gamma$ we have:
\[
hgh^{-1} = h\,uk\,h^{-1} = huh^{-1}\,\, hkh^{-1}\in huh^{-1}K=uK,
\]
since $\Gamma/K$ is abelian. We conclude that if $g\in uK$, then the whole conjugacy class of $g$ lies in $uK$. This proves that $uK$ contains infinitely many conjugacy classes of $\Gamma$.
\end{proof}

\section{Some conjectures}\label{sec5}

A careful analysis of known groups that have property $R_\infty$ and those that do not have it,
reveals an interesting pattern: all known finitely generated residually finite groups without property $R_\infty$ are solvable-by-finite. 
Fel'shtyn and Troitsky expressed this observation as the following  conjecture~\cite{FelTro2} (see also~\cite[Problem~19.28]{KouNot}):

\begin{conjR}
A finitely generated residually finite group either has property $R_\infty$ or is solvable-by-finite.
\end{conjR}
Note that the conjunction `or' in this statement is not exclusive, since there are many groups known to have property $R_\infty$ while being solvable-by-finite (e.g.\ the free nilpotent groups $N_{r,c}$ of rank $r$ and nilpotency class $c\ge2r$~\cite{DekGon1}). In a recent paper~\cite{Troit1} Conjecture (R) was proved for all groups with finite upper (Pr\"ufer) rank.

Since solvable-by-finite groups are amenable, a weaker variant of Conjecture (R)  would be
\begin{conjNA}
A finitely generated residually finite non-amenable group has property $R_\infty$.
\end{conjNA}
This statement appears as a theorem in the preprint~\cite{FelTro1}; however, Evgenij Troitsky has informed us that the proof contains a mistake, so we record it here as a conjecture.


\end{document}